\documentclass{siamltex}
\usepackage[latin1]{inputenc}
\usepackage[english]{babel}
\usepackage{graphicx}
\usepackage{amsmath,amsfonts,amssymb}
\usepackage{bbm,wasysym,stmaryrd}
\usepackage{subfigure}
\usepackage{color}
\usepackage{hyperref}
\usepackage[normalem]{ulem}

\newcommand{\hbx}{\hfill$\Box$}

\newcommand{\eps}{\varepsilon}

\newcommand{\R}{\mathbb{R}}
\newcommand{\N}{\mathbb{N}}

\newcommand{\W}{\mathcal{W}}
\newcommand{\I}{{\mathcal I}}

\renewcommand{\)}{\right)}
\newcommand{\der}[2]{ \frac{\text{d} #1}{\text{d} #2} }  

\newenvironment{remark} {\par {\noindent \it \sc Remark.} \small \it } {}

\title{Wild oscillations in a nonlinear neuron model with resets: (II)  Mixed-mode oscillations}

\author{Jonathan E. Rubin\footnotemark[1] \and Justyna Signerska-Rynkowska\footnotemark[6] \footnotemark[3] \footnotemark[4] \and Jonathan D. Touboul\footnotemark[6] \footnotemark[4]
\and Alexandre Vidal\footnotemark[5] \footnotemark[4] }
\renewcommand{\thefootnote}{\fnsymbol{footnote}}
\footnotetext[1]{Department of Mathematics, University of Pittsburgh, Pittsburgh, USA, jonrubin@pitt.edu}
\footnotetext[6]{The Mathematical Neuroscience Team, CIRB-Collège de France (CNRS UMR 7241, INSERM U1050, UPMC ED 158, MEMOLIFE PSL*), Paris, France}
\footnotetext[3]{Faculty of Applied Physics and Mathematics, Gda\'{n}sk University of Technology, Poland, \\jsignerska@mif.pg.gda.pl}
\footnotetext[4]{Mycenae, Inria, Paris, France}
\footnotetext[5]{Laboratoire de Math\'ematiques et Mod\'elisation d'\'Evry (LaMME), CNRS UMR 8071, Universit\'e d'\'Evry-Val-d'Essonne, France}

\begin{document}
\maketitle

\renewcommand{\thefootnote}{\arabic{footnote}}
\begin{abstract}
This work continues the analysis of complex dynamics in a class of bidimensional nonlinear hybrid dynamical systems with resets modeling neuronal voltage dynamics with adaptation and spike emission. We show that these models can generically display a form of mixed-mode oscillations (MMOs), which are trajectories featuring an alternation of small oscillations with spikes or bursts (multiple consecutive spikes). The mechanism by which these are generated relies fundamentally on the hybrid structure of the flow: invariant manifolds of the continuous dynamics govern small oscillations, while discrete resets govern the emission of spikes or bursts, contrasting with classical MMO mechanisms in ordinary differential equations involving more than three dimensions and generally relying on a timescale separation. The decomposition of mechanisms reveals the geometrical origin of MMOs, allowing a relatively simple classification of points on the reset manifold associated to specific numbers of small oscillations. We show that the MMO pattern can be described through the study of orbits of a discrete \emph{adaptation map}, which is singular as it features discrete discontinuities with unbounded left- and right-derivatives. We study orbits of the map via rotation theory for discontinuous circle maps and elucidate in detail complex behaviors arising in the case where MMOs display at most one small oscillation between each consecutive pair of spikes.
\end{abstract}


\noindent {\bf Running title.}  MMOs in a nonlinear neuron model

\newpage

\section{Introduction}

In this paper, we continue our study of hybrid integrate-and-fire neuronal models from \cite{paper1}, turning our attention to the analysis of mixed-mode oscillations (MMOs).  
MMOs are trajectories exhibiting small or subthreshold oscillations alternating with one or more large amplitude oscillations or spikes.  
These appear in a variety of cell types and brain areas including inferior olive nucleus neurons~\cite{bernardo-foster:86,llinas-yarom:81,llinas-yarom:86}, stellate cells
of the entorhinal cortex \cite{alonso-klink:93,alonso-llinas:89,jones:94,yoshida:07}, and neurons in the dorsal root ganglia \cite{amir-michaelis-etal:99,liu-michaelis-etal:00,llinas:88}, as well as in thalamocortical spindle waves~\cite{luthi1998periodicity}.  
Neurons transmit information  through the timing of spikes and the pattern of spikes fired, and  
subthreshold oscillations and associated MMO patterns may contribute to the precision, timing, and robustness of neuronal spiking \cite{lampl:93,yoshida:07,rotstein:encyc} as well as to spatial navigation \cite{giocomo:07}.
The major goal of this work is to provide a detailed mathematical analysis of MMOs in a class of neuronal models.  Specifically, in this article, we (i) consider a class of planar hybrid models widely used to model the electrical activity of neurons, (ii) show that models in this class are able to generate a wide range of MMO patterns, (iii) introduce a general mathematical framework for studying the dynamical structure involved and the orbits that result, and (iv) describe the geometric mechanism underlying these patterns. 

From the biological viewpoint, neuronal activity patterns, including MMOs, rely on ionic and biochemical mechanisms that are accurately described by nonlinear dynamical systems of relatively high complexity, such as variants on the celebrated Hodgkin-Huxley model~\cite{hodgkin-huxley:52,rubin2007giant}.
As described in our companion paper \cite{paper1}, in contrast with detailed biophysical models, integrate-and-fire models are abstractions of the voltage dynamics in which differential equations describing the dynamics of membrane depolarization are combined with a discrete reset corresponding to the emission of an action potential (a spike) and subsequent hyperpolarization. These models, first introduced more than a century ago~\cite{lapicque:07}, have evolved to incorporate nonlinearities to model the fast dynamics of spike initiation~\cite{brunel-latham:03,fourcaud-trocme-hansel-etal:03} and additional variables modeling adaptation~\cite{izhikevich:03}, synaptic dynamics~\cite{mihalas2009generalized} or resonant properties~\cite{izhikevich2001resonate}. Among these models, nonlinear bidimensional integrate-and-fire models with blow-up and resets are widely used in computational neuroscience, owing to their relative simplicity yet very rich dynamical phenomenology~\cite{brette-gerstner:05,izhikevich:03,izhikevich:07,touboul:08,touboul-brette:09}. However, none of these studies reported the presence of MMOs in these systems. 

More generally and despite their importance for applications, MMOs have so far received little attention in hybrid systems. A notable exception is the work of Rotstein and collaborators on linear bidimensional \emph{resonate-and-fire} neuron models~\cite{izhikevich2001resonate}. These models are organized around an unstable focus and naturally exhibit a variety of MMOs. Numerical simulations guided by characterizations of the trajectories and timescale analysis were used to explore associated subthreshold dynamics \cite{rotstein:06,rotstein:08} and to show how they can give an abrupt increase in firing frequency \cite{rotstein:13}. MMOs in this class of models result from a combination of subthreshold oscillations together with subsequent threshold crossings corresponding to spikes. The multi-timescale MMO scenario in these models does not necessarily represent the general mechanism for MMOs in hybrid models, however, and to date, there has not been a thorough analytical investigation of the detailed mechanisms underlying MMOs, incorporating both subthreshold and spiking components, in these models in the absence of timescale separation. 

In the current manuscript we present a rigorous study of MMOs in nonlinear bidimensional integrate-and-fire neuron models. We show that these can exhibit a wide variety of MMO patterns when the subthreshold dynamics features two unstable fixed points, a saddle and an unstable focus. We investigate spike patterns through iterates of a discrete map, the so-called \emph{adaptation map} introduced in~\cite{touboul-brette:09} (see the companion paper~\cite{paper1} for more details on the construction and use of this map). While previous works have considered settings in which the adaptation map is continuous, here, in the presence of an unstable focus, we will show that the adaptation map is singular: it may be undefined on a countable set of values at which the map has well-defined and finite left and right limits and infinite one-sided derivatives. A number of difficulties emerge from the irregular nature of the map; since associated circle maps may also feature analogous singularities, classical theories of Poincar\'{e} and Denjoy of circle homeomorphisms or their extensions to continuous non-invertible maps (\cite{misiu}) do not apply, and because of the unbounded derivative, neither do theories of discontinuous contractive maps~\cite{alseda,keener}. This contrasts with previous detailed studies of interspike intervals for periodically driven one-dimensional integrate-and-fire models \cite{coombes,gedeon,keener2,wmjs,js,tiesinga}. Here, we will demonstrate a fundamental relationship between the type of MMO pattern arising and the rotation number of the adaptation map. With the aim of characterizing rotation numbers of these maps, we build upon a number of theoretical results on circle maps that may have discontinuities~\cite{brette,alseda,keener,misiu,frrhodes,frrhodes2} and sometimes extend these to singular maps with unbounded derivative. In this way, we describe a new mechanism underlying robust MMOs, not requiring multiple timescales, in hybrid dynamical systems constituting an important class of neuron models.

Our presentation of these results is organized as follows. In section \ref{sec:Framework}, we introduce the model studied, review a few results on its dynamics, and describe the geometric mechanisms underlying the generation of MMOs.  We detail the properties of the adaptation map in section~\ref{sec:AdaptationMap}, with a particular focus on discontinuity points and divergence of the derivative, which is proved to be a general result based on a Poincar\'e section encompassing the stable manifold of a saddle. We further show that the particular structure of the map ensures that any type of transient MMO can be generated by these neuron models. In section~\ref{sec:OneDisc}, we use discontinuous rotation theory to develop a precise description of the dynamics in the case where the adaptation map admits one discontinuity in its invariant interval. Implications and perspectives in dynamical systems and neuroscience, as well as some extensions and prospects for analyzes of cases with more discontinuities, are discussed in sections~\ref{sec:MoreDisc} and~\ref{sec:Discussion}.

\section{Hybrid neuron model and the geometry of the MMO mechanism}\label{sec:Framework}
In this work, we study the class of integrate-and-fire neuron models introduced in~\cite{touboul:08}, described in detail in the companion paper \cite{paper1}: 
\begin{equation}\label{eq:SubthreshDyn}
  \begin{cases}
    \der{v}{t} = F(v) - w + I \\
    \der{w}{t} = \varepsilon (bv - w),
  \end{cases}
\end{equation}
where  $\eps,b>0$ and $I$ are real parameters.   Following \cite{touboul:08, touboul-brette:09}, we will assume that the real function $F$ is regular (at least three times continuously differentiable), strictly convex, superquadratic at infinity, with its derivative having a negative limit at $-\infty$ and an infinite limit at $+\infty$ (see Assumption (A1) in the companion paper \cite{paper1}). These assumptions imply in particular that 
the membrane potential blows up in finite time and at this explosion time, say $t^*$, the adaptation variable converges to a finite value $w(t^{*-})$ \cite{touboul:09}.  At time $t^*$, it is considered that the neuron has fired a spike; the voltage is instantaneously reset to the fixed reset value $v_R$ and the adaptation variable is updated as follows:
\begin{equation}\label{eq:new_reset}
	\begin{cases}
		v(t^*)=v_R\\
		w(t^*)=\gamma w(t^{*-})+ d
	\end{cases}
\end{equation}
with $\gamma \leq 1$ and $d\geq 0$; in this work, we allow $\gamma < 1$, which arises in accounting for spike duration (see \cite{paper1} for details).

Numerical simulations performed in the present article correspond to the case of the quartic model $F(v)=v^4+ a v$ with, unless otherwise stated, 
\begin{equation} \label{Par_Val}
a=0.2, \quad \varepsilon=0.1, \quad b=1, \quad I=0.1175, \quad v_R=0.1158.
\end{equation}
The values of the parameters associated with resets, $d$ and $\gamma$, are left free and will be used as  bifurcation parameters. 

As discussed in the companion paper \cite{paper1},  the 1-dimensional \emph{adaptation map} $\Phi$ can be defined based on  the orbits of the system in the phase plane. Specifically, if $(V(\cdot,v_R,w),W(\cdot,v_R,w))$ is the solution of equation~\eqref{eq:SubthreshDyn} with initial condition $(v_{R},w)$ and if $V$ blows up at $t^{*}$ (i.e., $\lim_{t\to t^{*-}}V(t,v_R,w)=\infty$), then 
\begin{equation}
\label{eq:admap}
\Phi(w):=W(t^{*},v_R,w)=\gamma W(t^{*-},v_R,w)+d
\end{equation}
 is the associated value of the adaptation variable after spike and reset. 
In \cite{paper1,touboul-brette:09}, we detailed the mathematical analysis of \eqref{eq:SubthreshDyn}-\eqref{eq:new_reset} in the absence of fixed points (yellow region of Fig.~1.1 of \cite{paper1}). In that regime, the system fires an action potential for any initial condition in the phase plane. 
The adaptation map is thus well-defined on the whole real line and smooth. The orbits of the map can be used to analyze spike patterns and transitions between them.

In the present manuscript, we analyze the  dynamics in regions in which the system features an unstable spiral point and a saddle (pink region of Fig. 1.1 of \cite{paper1}).  The stable manifold of the saddle is a 
one-dimensional heteroclinic orbit spiraling out from the unstable focus (see Fig.~1.1 of \cite{paper1} and Fig.~\ref{fig:Structure}). This geometry of the phase space constrains trajectories reset within the spiral to proceed to a prescribed number of rotations around the unstable fixed point before firing (Fig.~\ref{fig:Structure}, inset and bottom). The rotations around the unstable fixed point provide small oscillations used to define mixed-mode oscillations as follows.

\begin{definition} Mixed-mode oscillations (MMOs) for the system \eqref{eq:SubthreshDyn}-\eqref{eq:new_reset} are spiking orbits consisting of an alternation of small oscillations and spikes. MMO patterns formed by a sequence of $\mathcal{L}_{k} \in \N^{*}:=\{1,2,...\}$ spikes followed by $s_{k}\in \N^{*}$ small oscillations are characterized by their \emph{signature} $\mathcal{L}_1^{s_1} \mathcal{L}_2^{s_2} \mathcal{L}_3^{s_3} \cdots$. Periodic signatures with period $k$ are only denoted by finite sequence of length $k$, $\mathcal{L}_1^{s_1} \mathcal{L}_2^{s_2} ... \mathcal{L}_k^{s_k}$.
\end{definition}

\begin{remark} 
\begin{itemize}
\item MMOs featuring bursts of two or more consecutive spikes not separated by periods of small oscillations (i.e., $\mathcal{L}_{k}\geq 2$ for some $k$) are referred to as mixed-mode bursting oscillations (MMBOs). We use the term MMO as a generic term to describe any combination of spikes and small oscillations, and the term MMBO is applied specifically to distinguish those trajectories featuring bursts and small oscillations. 
\item In the present paper, we will be able to distinguish small oscillations at half-rotation precision, and thus will extend the definition above to signatures with half-integer number of small oscillations $s_{k}\in \frac 1 2 \N^{*}$. 
\end{itemize}
\end{remark}

We henceforth assume that the reset line $\{v=v_R\}$ intersects the spiraling stable manifold of the saddle, as in Fig.~\ref{fig:Structure}. The adaptation map is undefined at each intersection of the reset line with the stable manifold of the saddle, since the orbit of \eqref{eq:SubthreshDyn} starting from such a point converges to the saddle, and thus no spike follows.  For any initial condition $(v_R,w)$  not on the stable manifold, the associated orbit performs a specific number of small oscillations before firing, resulting in an MMO pattern. As indicated in Fig.~\ref{fig:Structure}, the present framework allows us to perform a detailed analysis of this scenario, since:
\begin{itemize}
	\item the fact that the stable manifold is bounded in the $v$ variable implies that the amplitudes of small oscillations, similarly to biological MMOs, are  considerably smaller than the spike amplitude, and 
	\item the intersections of the stable manifold with the reset line partition the values of $w$ associated with a specific number of small oscillations (with half-rotation precision). 
\end{itemize}
The signature of the MMO patterns can be deduced from a dynamical analysis of the adaptation map.
The main objective of the manuscript is  to characterize these patterns, and the main results are summarized below.

\begin{figure}[htbp]
	\centering
		\includegraphics[width=\textwidth]{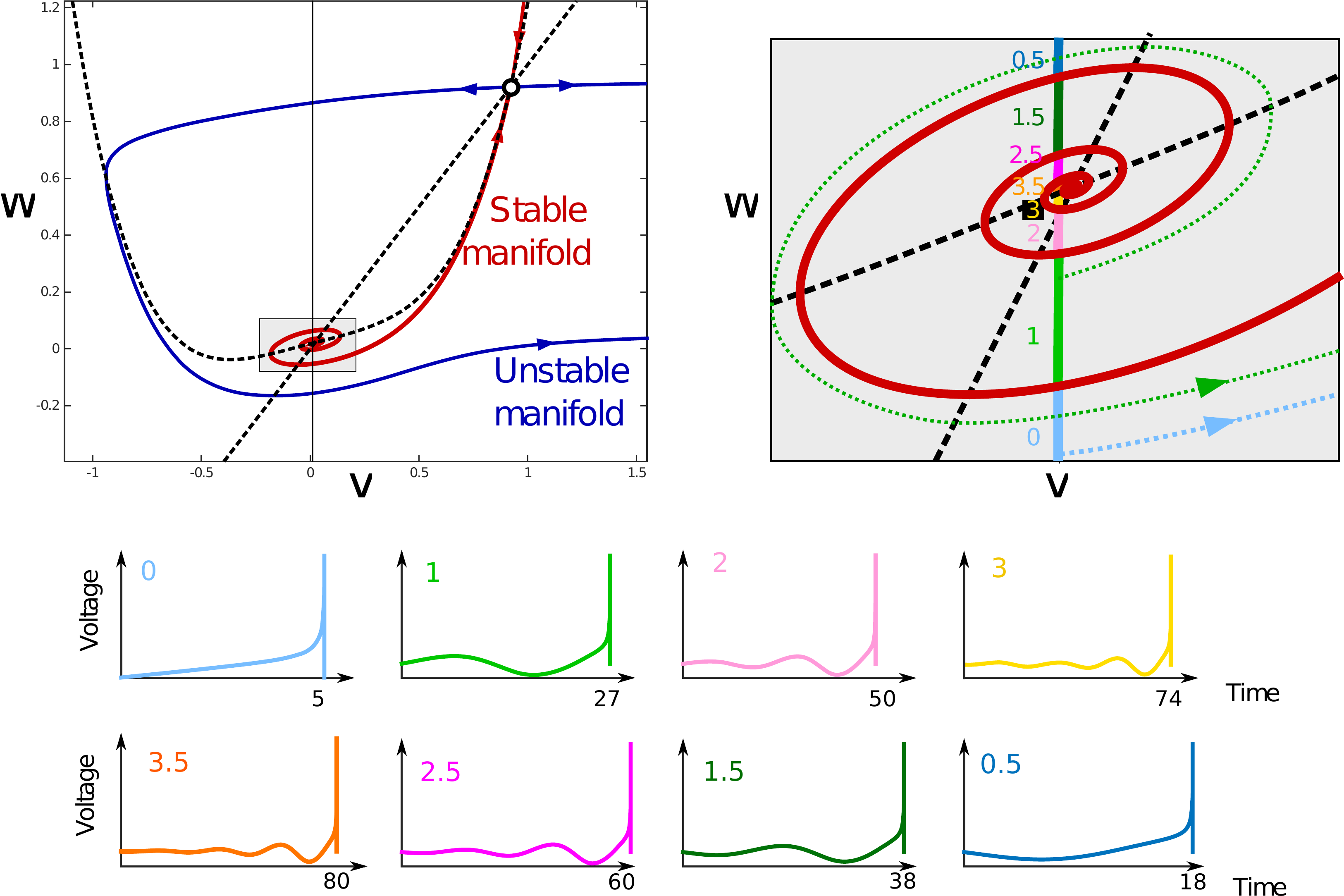}
	\caption{The geometry of MMOs: (upper row) Phase plane with $v$ and $w$ nullclines (dashed black) and stable (red) and unstable (blue) manifolds of the saddle; the stable manifold winds around the repulsive singular point. The reset line $\{v=v_R\}$ (solid vertical line) intersects the stable manifold, separating out regions such that trajectories emanating from  each undergo a specific number of small oscillations (colored segments, here from 0 to 3 below the $w$-nullcline and from $3.5$ to $0.5$ above). (Lower rows) The solution for one given initial condition in each segment. Note that the time interval varies in the different plots (indicated on the $x$-axis). Simulations had initial conditions $v=v_R=0.012$ and $w$ chosen within the different intervals on the reset line.}
	\label{fig:Structure}
\end{figure}

We establish that, as a transient behavior, the system can feature MMOs with all possible finite signatures (Proposition \ref{Newallpatterns} and Corollary \ref{Newallpatterns2}). Non-transient behaviors are deduced from the iterates of the adaptation map, which may feature several discontinuities and therefore support a very wide range of possible dynamics. We concentrate in section~\ref{sec:OneDisc} on the case where the adaptation map features a single discontinuity within its invariant interval. As in the seminal study of Keener on maps with one discontinuity~\cite{keener}, we distinguish two cases depending on the monotonicity of the lift, called overlapping or non-overlapping cases. In the non-overlapping case (see subsection \ref{nonoversec}), we characterize the rotation number of the associated adaptation map and show that it characterizes the MMO signature (Theorem \ref{dynamics2}) or the chaotic nature of the spike pattern fired. In the overlapping case (see subsection \ref{sec:overlap}), the adaptation map yields rotation intervals with rational numbers corresponding to periodic orbits with MMOs (Proposition  \ref{Effective}). To go beyond this description, we provide conditions that guarantee existence of periodic orbits with arbitrary periods, all displaying MMBOs (Proposition \ref{male1}). 
Eventually, we discuss how the methods used here could be extended to cases with multiple discontinuity points within the invariant interval of the adaptation map (Theorem \ref{GeneralOldHeavy}).

\section{The adaptation map}\label{sec:AdaptationMap}
We start by characterizing the properties of the adaptation map $\Phi$ given in (\ref{eq:admap}). In the absence of singular points of the subthreshold dynamics, it is defined and continuous on $\R$, and the nature of its orbits distinguishes regular spiking (fixed point of the map), bursting (periodic orbit of the map) or chaotic spiking~\cite{paper1,touboul-brette:09}. In the present case with two singular points (unstable focus and saddle), we show that $\Phi$ is undefined at specific points, no longer continuous and has unbounded derivative, but its orbits still provide all the information necessary to characterize the associated MMO patterns. 

\subsection{Properties of the adaptation map}\label{subsec:AdaptationMap}

 Throughout the manuscript, we assume that  the vector field \eqref{eq:SubthreshDyn} has two unstable singular points, the repulsive singular point $(v_-,w_-)=(v_-,F(v_-)+I)$ and the saddle singular point $(v_+,w_+)=(v_+,F(v_+)+I)$, with $v_-<v_+$ (see~\cite{touboul:08,touboul-brette:09} providing detailed bifurcation analysis of the subthreshold dynamics). We denote by $\W^s$ and $\W^u$ the stable and unstable manifolds of the latter singular point; each of these is decomposed into two branches (Fig. \ref{fig:ImportantPoints}) with $\W^s_-$ ($\W^s_+$) extending towards $w<w_+$ ($w>w_+$) and $\W^u_-$ ($\W^u_+$) extending towards $v<v_+$ ($v>v_+$). 
 The shape of the map $\Phi$ is organized around a few important points (see Fig.~\ref{fig:ImportantPoints}):
\begin{itemize}
	\item We denote by $w^*=F(v_R)+I$ the intersection of the reset line $v=v_R$ with the $v$-nullcline.
	\item We denote by $w^{**}=b v_R$ the intersection of the reset line with the $w$-nullcline.
	\item We denote by $\{ w_i\}_{i=1}^{p}$ the sequence of intersections of the reset line with $\W^s$, labeled in increasing order with respect to the value of $w$. As long as $v_R \neq v_-$, there exists a finite number of such points or none depending on the parameter values: an even number of intersections for $v_R<v_-$ and an odd number for $v_R>v_-$. We denote by $p_1$ the index such that $\{ w_i\}_{i\leq p_1}$ are below the $v$-nullcline and $\{ w_i\}_{i>p_1}$ are above; it is easy to see that $p_1$ is the smallest integer larger than or equal to $p/2$, i.e. $p_1=\lceil p/2 \rceil$. The points $\{ w_i\}$ split the real line into $p+1$ intervals that we denote $\{ I_i\}_{i=0}^{p}$, with $I_0=(-\infty,w_1)$, $I_i=(w_i,w_{i+1})$, $i=1,2,...,p-1$, and $I_p=(w_p,\infty)$. Remark that these intervals precisely correspond to those in which the number of small oscillations occurring between two consecutive spikes is constant, except the interval $I_{p_1}$, which is split into two subintervals by $w^*$ (see Fig.~\ref{fig:Structure}). The number of small oscillations for trajectories starting from $I_i$ is
	\begin{equation}\label{nosmall}
	\begin{cases}
	 i &\text{ if } i< p_1, \\
	 (p+1/2)-i &\text{ if } i>p_1,\\
	 p_1 &\text{ if } i=p_1 \text{ and } w<w^*, \\
	 p_1+1/2 &\text{ if } i=p_1 \text{ and } w>w^* \text{ and $p$ is even}\\
     p_1-1/2 &\text{ if } i=p_1 \text{ and } w>w^* \text{ and $p$ is odd}.
	\end{cases}
	\end{equation}

	\item We denote by $w_{\lim}^-<w_{\lim}^+<\infty$ the limit of the adaptation variable when $v \to  +\infty$ along $\W^u_-$ and $\W^u_+$ respectively. In addition, we introduce the corresponding values obtained through the reset mechanism:
	\[
	\beta=\gamma w_{\lim}^- +d, \quad \alpha =\gamma w_{\lim}^+ +d.
	\]
\end{itemize}

\begin{figure}[htbp]
	\centering
	\includegraphics[width=.8\textwidth]{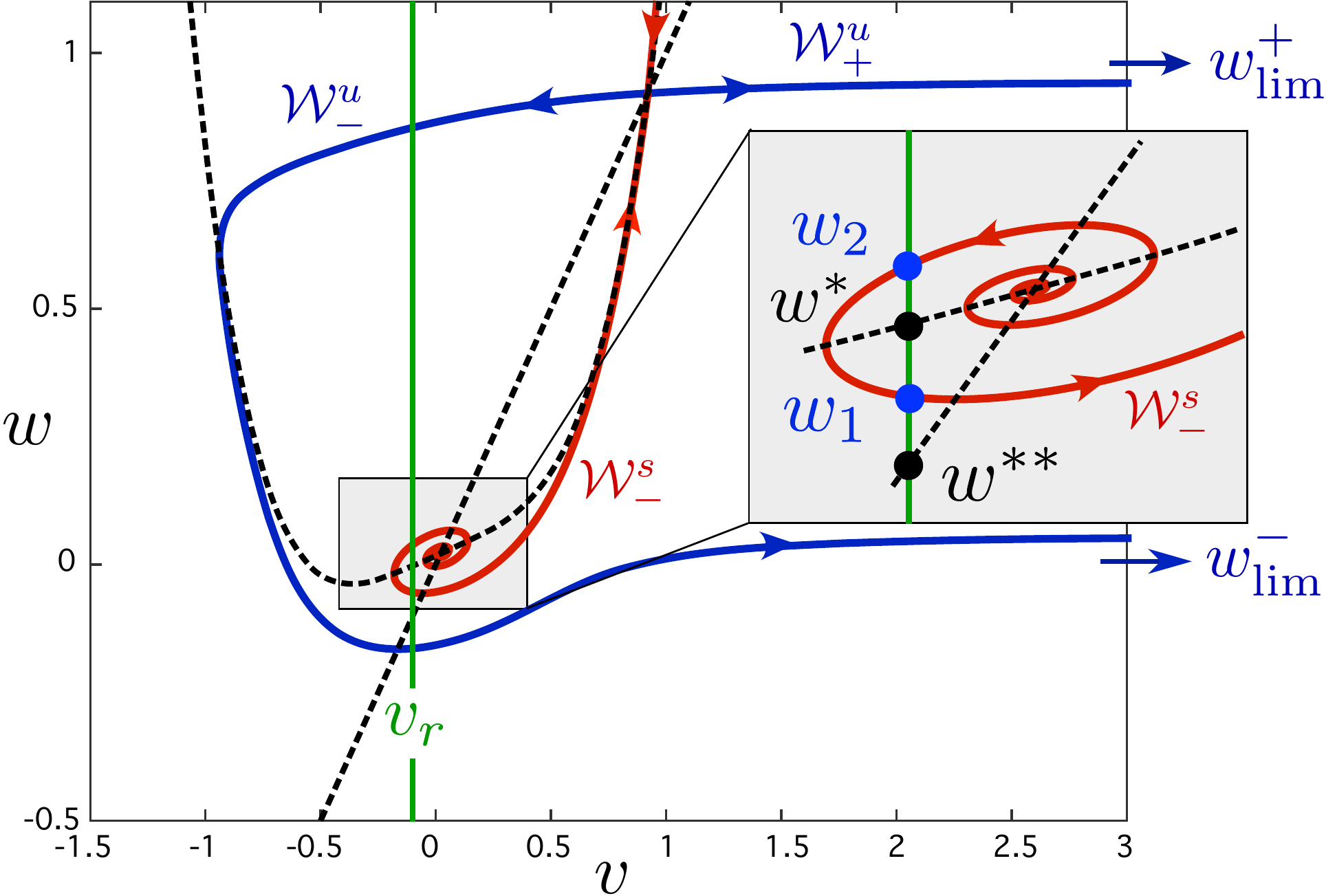}
	\caption{Geometry of the phase plane with indication of the points relevant in the characterization of the adaptation map $\Phi$. In this example, there are only $p=2$ intersections of $\{ v=v_R\}$ with $\W^s$ (thus $p_1=1$).}
	\label{fig:ImportantPoints}
\end{figure}

With these points defined, we can characterize the shape of the adaptation map. When we refer to the adaptation map, we abuse notation and use $w_i$ to denote the $w$-coordinates of the points of intersection of $W^s$ with $\{ v=v_R \}$. 
\begin{theorem}\label{prop_phi}
The adaptation map $\Phi$ has  the following properties.
\begin{enumerate}
  \item It is defined for all $w\in\mathcal D := \mathbb{R}\setminus \{w_i\}_{i=1}^p$.
  \item It is regular (at least $C^3$) everywhere except at the points $\{w_i\}_{i=1}^p$.
  \item In any given interval $I_i$ with $i\in\{0,\cdots, p\}$, the map is increasing for $w<w^*$ and decreasing for $w>w^*$.
  \item At the boundaries of the definition domain $\mathcal D$, $\{w_i\}_{i=1}^p$, the map has well-defined and distinct left and right limits:
  \[
	\begin{cases}
		  \lim\limits_{w\to w_i^{-}}\Phi(w)=\alpha \text{ and } \lim\limits_{w\to w_i^{+}}\Phi(w)=\beta \text{ if } i\leq p_1, \\
		  \lim\limits_{w\to w_j^{-}}\Phi(w)=\beta \text{ and } \lim\limits_{w\to w_j^{+}}\Phi(w)=\alpha \text{ if } j>p_1.
	\end{cases}
  \]
  \item The derivative $\Phi^{\prime}(w)$ diverges at the discontinuity points\footnote{With a slight abuse of terminology, we refer to the points $w_i$ as \emph{discontinuity points} although $\Phi$ is formally not defined at $w_i$.}:
  \[
  	\begin{cases}
  		\lim\limits_{w\to w_i^{\pm}}\Phi^{\prime}(w)=+\infty \text{ if }  i\leq p_1,\\
		\lim\limits_{w\to w_i^{\pm}}\Phi^{\prime}(w)=-\infty \text{ if } i> p_1.
  	\end{cases}
	\]
  \item $\Phi$ has a horizontal plateau for $w\to+\infty$ provided that
  \begin{equation}\label{plat_cond}
  \lim\limits_{v\to-\infty}F^{\prime}(v)<-\eps(b+\sqrt{2}).
  \end{equation}

  \item For $w<\min\left(\frac{d}{1-\gamma},w_1,w^{**}\)$, we have $\Phi(w)\geq \gamma w+d>w$.
  \item If $v_R<v_+$, $\Phi(w) < \alpha$ for all $ w\in \mathcal D$. Moreover, for any $w$ taken between the two branches of the unstable manifold of the saddle, hence in particular for $w\in (w_1,w_p)$, $\Phi(w)>\beta$.

\end{enumerate}
\end{theorem}

In comparison to the case without singular points~\cite[Theorem 3.1]{touboul-brette:09}, the map loses continuity, convexity, and uniqueness of the fixed point, but the monotonicity property (point 3), the presence of a plateau (point 6) and the comparison with identity (point 7) remain true. 
The presence of discontinuities and divergence of the map  derivative substantially change the nature of the dynamics as we will see below. It is worth noting that this divergence is a general property of correspondence maps in the vicinity of saddles (see Fig.~\ref{fig:label}), as we show in the following:

\begin{lemma}\label{lem:infiniteDerivative}
	Consider a two-dimensional smooth vector field (at least $C^2$) with a hyperbolic saddle $x_0$ associated with the eigenvalues $-\mu<0<\nu$ of the linearized flow. We denote by $W^s$ and $W^u$ the stable and unstable manifolds of the saddle and consider two transverse sections $S_s$ and $S_u$ intersecting $W^s$ and $W^u$ at $x_s$ and $x_u$, respectively. There exists $\Omega_s$ a one-side neighborhood of $x_s$ on $S_s$ that maps onto a one-side neighborhood $\Omega_u$ of $x_u$ on $S_u$. The correspondence map $\Psi:\Omega_s \mapsto \Omega_u$ is differentiable in $\Omega_s$ and we denote by $\Psi'$ the one-sided differential of $\Psi$ at $x_s$. We have:
	\begin{equation}\label{eq:differential}
		\Psi'=
		\begin{cases}
			\infty & \text{ if } \nu-\mu>0\\
			0 & \text{ if }\nu-\mu<0
		\end{cases}
	\end{equation}
	When $\mu=\nu$, the differential is finite and its value depends on the precise location of the sections.
\end{lemma}

\begin{figure}
	\centering
		\includegraphics[width=.7\textwidth]{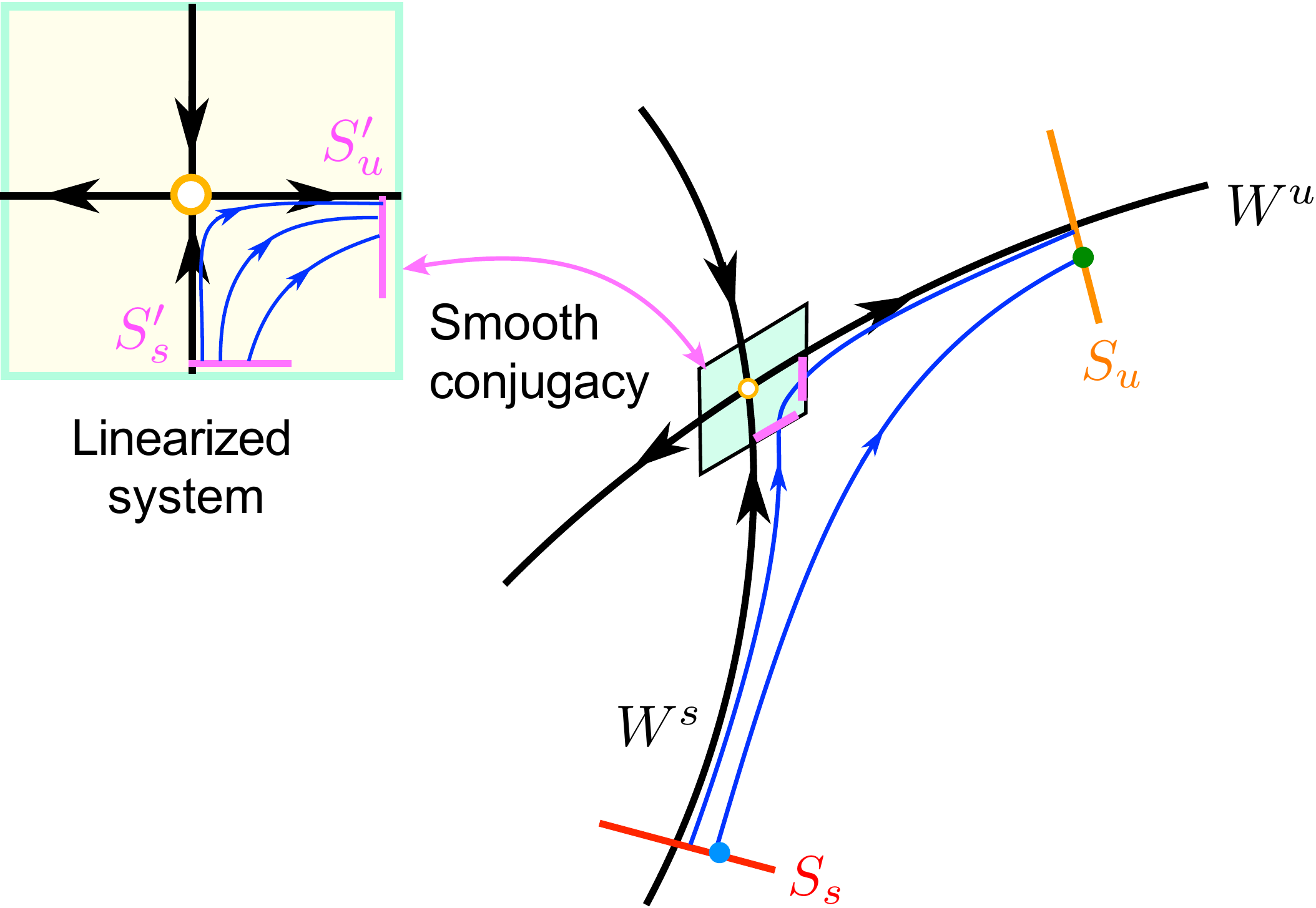}
	\caption{Typical topology of manifolds and sections in Lemma~\ref{lem:infiniteDerivative}: we consider correspondence map between sections $S_s$ (red) and $S_u$ (orange) transverse, respectively,  to the stable and unstable manifolds (black lines) of the saddle (orange circle). Typical trajectories are plotted in blue. The key arguments are the characterization of correspondence maps associated with the linearized system (upper left inset) between two transverse sections $S_s'$ and $S_u'$, and the smooth conjugacy between the nonlinear flow and its linearization.}
	\label{fig:label}
\end{figure}

\begin{proof}
	Let us start by considering the linearized system in the vicinity of the saddle singular point. In the basis that diagonalizes the Jacobian, we can write the system in the simple form:
	\[\begin{cases}
		\dot x = -\mu x\\
		\dot y = \nu y
	\end{cases}\]
	and considering a section $S_s'$ corresponding to $y=y_0$ and a section $S_u'$ corresponding to $x=x_0>0$, simple calculus leads to the formula that the correspondence map $\varphi$ of the linearized system between $S_s'$ and $S_u'$ is defined for $x\geq 0$ by $\varphi(x)=(y_0/x_0^{\xi})x^{\xi}$ with $\xi=\mu/\nu$. Hence, the differential of $\varphi$ at $0^+$ is such that:
	\begin{itemize}
		\item it diverges if $\xi<1$, hence for $\nu-\mu>0$ (i.e. if the dilation along the unstable direction is stronger than the contraction along the stable direction);  
		\item it vanishes if $\xi>1$, hence for  $\nu-\mu<0$ (i.e. if the contraction along the stable direction is stronger than the dilation along the unstable direction); 
		\item when $\xi=1$ (i.e. contraction and dilatation are of the same intensity), we find $\varphi'(0^+) = y_0/x_0$ which depends on the precise location of the sections.
	\end{itemize}
To demonstrate the lemma, we thus need to show that the same result holds for the nonlinear system. The Hartman-Grobman Theorem~\cite{hartman1963local}, which ensures that the nonlinear system is conjugated to its linearization through a homeomorphism in the vicinity of the (hyperbolic) saddle, will not be sufficient; we need to ensure that the nonlinear and linear flows are locally conjugated via \emph{smooth diffeomorphisms} (at least $C^1$). Finding smooth conjugacy is a subtle question for a general dynamical system that has been the object of significant research and generally requires avoiding resonances in the eigenvalues, which may lead to a relatively complex relationship~\cite{samovol1989necessary,sternberg1957local}. In two dimensions, the problem is simpler and it was proved in  \cite{stowe1986linearization} that any $C^2$ planar dynamical system in the neighbourhood of a saddle is smoothly (with at least $C^1$ regularity) conjugated with its linearization, and the derivative of this conjugacy is bounded away from $0$ in a sufficiently small neighborhood of the saddle (since this conjugacy converges in a $C^{1}$-sense towards the identity close from the saddle).  Completing the proof thus only amounts to showing that the correspondence maps from a neighborhood of $S_s$ to $S_s'$ and from $S_u'$ to a subset of $S_u$ are smooth with differential bounded away from zero and infinity. This is a classical consequence of the flow box theorem and regularity with respect to the initial condition.
\end{proof}

Now that this general result is proved, we proceed to establish the properties of the adaptation map by proving Theorem~\ref{prop_phi}.

\begin{proof} \emph{of Theorem \ref{prop_phi}}: 
First, note that generalization of the reset mechanism by introducing a parameter $\gamma\in (0,1]$ does not substantially impact the shape of the adaptation map. Indeed, all properties rely on the map associating with a point on the reset line $(v_R,w)$ the value $\varphi(w):=W(t^{*-};v_R,w)$ of the adaptation variable at the time $t^*$ of the subsequent spike, since $\varphi(w)=(\Phi(w)-d)/\gamma$. In other words, the generalization of the reset mechanism does not introduce any new mathematical difficulty. Hence, the proofs for items 1. to 3. and 6. to 8. are straightforward extensions of the analogous proof in~\cite{touboul:08} or simple algebra. Similarly, the proof of property 4. follows reasoning analogous to that of 3., with left and right limits found by finely characterizing the shape of the trajectories as $w$ approaches one of the discontinuity points $w_i$. In all cases, the trajectory will initially remain very close to the stable manifold, before leaving the vicinity of the stable manifold near the saddle and following the unstable manifold. Depending on whether the trajectory approaches the saddle from the right or from the left, it will either follow the left or right branch of the unstable manifold, hence either converge towards $w_{\lim}^+$ or $w_{\lim}^-$. 
	
We focus on the proof of property 5., which requires specific analysis. We use Lemma~\ref{lem:infiniteDerivative} and prove that the conditions on the contraction and dilation near the saddle are satisfied. We consider the specific sections that define $\Phi$, namely $S_s=\{v=v_R\}$ (which is valid as long as the stable manifold is not tangent to the reset line) and a section corresponding to spiking, denoted with a slight abuse of notation as $S_u=\{v=+\infty\}$. The use of a section at $\infty$, however, does not exactly fit the statement of  Lemma~\ref{lem:infiniteDerivative}, and requires us to show that the differential of the correspondence map does not vanish as $v \to \infty$.

	First, note that since $(v_-,F(v_-)+I)$ is an unstable focus, the linearized flow there has two complex conjugate eigenvalues with positive real part and therefore the trace of the Jacobian, given by $F'(v_-)-\eps$, is strictly positive. Since $F$ is convex, the trace of the Jacobian at the saddle equals $F'(v_+)-\eps \geq F'(v_-)-\eps>0$. Hence the dilation at the saddle is always stronger than the contraction; in the notation of (\ref{eq:differential}), we have $\nu-\mu>0$.

To show that the infinite derivative persists when one considers $S_u=\{v=+\infty\}$, we express the map $\Phi$ formally in the region below the stable manifold of the saddle, which all spiking trajectories cross. In this region, any trajectory has a monotonically increasing voltage (that blows  up in finite time), and the orbit with initial condition $(v_R,w_0)$ can be expressed as the parametric curve $(v,W(v))$ with
	\begin{equation}\label{eq:Traj_recalled}
	\left\{
	  \begin{array}{ll}
	    \displaystyle{\frac{\mathrm{d}W}{\mathrm{d}v}=\frac{\eps\big(bv-W\big)}{F(v)-W+I} }, \\
	    W(v_R)=w_0.
	      \end{array}
	\right.
	\end{equation}
The expression of the differential of $W$ with respect to $w_0$ at $v$ is given by:
	\begin{equation}\label{eq:implicit derivative}
	\frac{\partial W}{\partial w_0}(v)= 1+\int_{v_R}^v\left(\frac{\eps(bu-F(u)-I))}{(F(u)-W(u)+I)^2}\right) \frac{\partial W(u)}{\partial w_0}\mathrm{d}u,
	\end{equation}
	with solution given, as a function of the trajectory $W$, by (see Peano's Theorem in \cite{phartman})~\footnote{From this expression one can propose an alternative direct (but particular) proof of the divergence of the one-sided (left) derivative at the points $w_i$ that does not rely on the general result of Lemma~\ref{lem:infiniteDerivative}. Indeed, the stable manifold $\W^s(u)$ has, close to $(v_+)^-$, a linear expansion
	$F(u)+I-\W^s(u) \sim -K(v_+ -u)$ with $K=\frac{1}{2} \left( \eps+F'(v_+)+\sqrt{(\eps+F'(v_+))^2-4\eps b}\right)>0$ and it is easy to deduce the divergence of the integral term within the exponential when $v\to (v_+)^-$.}:
	\begin{equation}\label{formula_for_derivative}
	\frac{\partial W}{\partial w_0}(v)=\exp\left(\int_{v_R}^{v}\frac{\eps(bu-F(u)-I)}{(F(u)-W(u)+I)^2}\mathrm{d}u\right).
	\end{equation}
	Hence, for any section $S_u=\{v=\theta\}$ (with $\theta<\infty$), the derivative of the map $\Phi$ cannot vanish. Furthermore, for $u$ large, we know that $W(u)$ remains finite and thus the integrand in~\eqref{formula_for_derivative} behaves as $-\eps/F(u)$ which is integrable at infinity (cf. Assumption (A1), \cite{paper1}). Consequently, the integral within the exponential term does not diverge towards $-\infty$ as $v\to\infty$ and the derivative (\ref{formula_for_derivative}) does not vanish at $S_u=\{v=\infty\}$. We further note that all correspondence maps away from singularities $(v_{-},F(v_{-})+I)$ and $(v_{+},F(v_{+})+I)$ are well-defined and with finite derivative bounded away from zero for the same reason. The intervals $(-\infty,w_1)$, $\{ I_i\}_{i=1}^p$, and $(w_p,\infty)$ on the line $\{v=v_R\}$ are transverse sections of the flow and correspondence maps from $I_i$ to $(-\infty,w_1)$ are increasing  for $i\leq p_1$ (hence the left and right differentials of $\Phi$ at $w_i^\pm$ for $i\leq p_1$ are equal to $+\infty)$ and decreasing otherwise (hence the left and right differentials at $w_i^\pm$ for $i>p_1$ are equal to $-\infty$). 
%
\end{proof}

\subsection{Transient MMO behaviors} \label{subsec:MMO_Signature}
We recall that at each discontinuity point $w_i$, the right and left limits of the adaptation map are always equal to either $\alpha$ or $\beta$. This property, related to the fact that all discontinuities correspond to intersections of the reset line with the stable manifold of the saddle, is a very important property that endows the system with a rich phenomenology, ensuring that it can generate MMOs of any signature.

We start by treating the case where the adaptation map has an infinite number of discontinuity points, which occurs  
in particular\footnote{This case also arises when the subthreshold dynamics \eqref{eq:SubthreshDyn} has a stable fixed point with a circular attraction basin bounded by an unstable limit cycle (orange region C in Fig.1.1 of~\cite{paper1}),  and $\{v=v_R\}$ intersects this limit cycle. This scenario involves a different fixed point structure than what we assume in this paper but the results on transient MMOs remain valid and the statements of further sections on asymptotic MMOs have their counterparts in this case.} when the reset line $\{v=v_R\}$ intersects the unstable focus $(v_-,F(v_-)+I)$.
We denote by $\{ m_i\}_{i \in \N^*}$ the $w$ values of the discontinuity points below the intersection $w^*$ of the reset line with the $v$-nullcline, with $m_i<m_{i+1}$ for any $i$. Similarly, we denote by $\{ M_i\}_{i \in \N^*}$ the $w$ values of the discontinuity points satisfying $M_i>w^*$ and $M_i>M_{i+1}$. We note that the left and right limits of $\Phi$ at $m_i$ ($M_i$) are well defined, equal to $\alpha$ and $\beta$, respectively ($\beta$ and $\alpha$, respectively)\footnote{Here the discontinuity points of $\Phi$ are denoted by $m_i$ and $M_i$, instead of $w_i$ as earlier, since we have two infinite sequences of intersections lying, respectively, below and above $w^*$ and we need to distinguish between them.}.

\begin{proposition}\label{Newallpatterns}
	Assume that the reset line has an infinite number of intersections with the stable manifold of the saddle. If moreover all the discontinuity points $\{ m_i\}_{i\in\mathbb{N}^*} \cup \{ M_i\}_{i\in\mathbb{N}^*}$ belong to $[\beta,\alpha]$, then for every $n\in \mathbb{N}^*$ and every finite sequence $\{ s_i\}_{i=1}^n$, where $s_i=k_i+l_i/2$, $k_i\in \mathbb{N}^*, \ l_i\in \{0,1\}$, there exists a set $J\subset [\beta,\alpha]$  with non-empty interior such that for any $w_0\in J$, the orbit with initial condition $(v_R,w_0)$ has a transient signature
\[
1^{s_1} 1^{s_2} 1^{s_3} ... 1^{s_n}.
\]
\end{proposition}

\begin{proof} We recall that for $w\in (m_k,m_{k+1})$ (resp. $w\in (M_k,M_{k+1})$), the orbit passing through $(v_R,w)$ performs exactly $k$ (resp. $k+1/2$) small oscillations before spiking. Thus, proving the proposition amounts to finding a set of initial conditions with a prescribed topological dynamics. In detail, given an MMO pattern $1^{s_1}1^{s_2} \ldots 1^{s_n}$, where the $s_i$ are as above, we are searching for sequences of iterates of $\Phi$ falling sequentially in the intervals 
	\[I_{s_i}:=\begin{cases}
		(m_{k_i},m_{k_i+1}),  & \mbox{if} \; l_i=0 \\
		(M_{k_i},M_{k_i+1}), & \mbox{if} \;  l_i=1.
	\end{cases}\]
	The set of initial conditions corresponding to this prescribed signature is therefore exactly $\Phi^{-1}(I_{s_1})\cap \Phi^{-2}(I_{s_2})\cap ...\cap \Phi^{-n}(I_{s_n})$, and proving the theorem amounts to showing that this set is not empty, which relies on the particular shape of the map $\Phi$ and specifically on the fact that $\Phi(I_{s_k}) = (\beta,\alpha)$ for any admissible $s_k$.
This property implies that for any interval $\mathcal{J}\subset (\beta,\alpha)$ with non-empty interior and any admissible $s_{k}$, the intersection of the pre-image $\Phi^{-1}(\mathcal{J})$ with $I_{s_{k}}$ is an interval with non-empty interior. In turn, this fact allows us to establish the proposition by recursion on the length of the transient signature $n$. Indeed, for $n=1$, the set of initial conditions associated to a transient signature $1^{s_{1}}$ is the set $\mathcal J_{1}=\Phi^{-1}(I_{s_{1}})$, which has a non-empty interval of intersection with both $(m_k,m_{k+1})$ and $(M_k,M_{k+1})$. Let us now assume that the same property is true for some $n\in \mathbb{N}^{*}$, namely that for any sequence $\{ s_i\}_{i=1\cdots n}$, there exists a set of initial conditions with a non-empty  interval of intersection with both $(m_k,m_{k+1})$ and $(M_k,M_{k+1})$ from which trajectories have the transient signature $1^{s_{1}}\cdots 1^{s_{n}}$. Let us now fix a sequence $\{ s_i\}_{i=1 \cdots n+1}$. By the recursion assumption, the set $\mathcal{J}_{n}$ associated to the transient signature $1^{s_2} 1^{s_3} ... 1^{s_{n+1}}$ is such that $\mathcal{J}_{n} \cap I_{s_{1}}$ is an interval with non-empty interior. Consequently, the set $\mathcal{J}_{n+1}=\Phi^{{-1}}(\mathcal{J}_{n} \cap I_{s_{1}})$ contains a non-empty interval of intersection with all $I_{s_{k}}$ and any trajectory with initial condition within that interval has the transient signature $1^{s_1} 1^{s_2} 1^{s_3} ... 1^{s_n}1^{s_{n+1}}$.\end{proof}

\begin{remark}
We emphasize that Proposition \ref{Newallpatterns} does not assure the presence of bursts of activity (i.e. $s_i\in \{0,1/2\}$): this is due to the fact that  $\Phi(\,(\beta,m_1)\,)=(\Phi(\beta),\alpha)$ and $\Phi(\,(M_1,\alpha)\,)=(\Phi(\alpha),\alpha)$ might be proper subintervals of $(\beta,\alpha)$ and  the argument used in the proof no longer applies. Therefore to account for $s_i=0$ or $s_i=1/2$ one would need to make an additional technical assumption.
\end{remark}

	If the reset map has a finite number of discontinuities, exactly the same proof applies to show that any MMO pattern with an accessible number of small oscillations exists. This extension  is precisely summarized in the following result.

\begin{corollary}\label{Newallpatterns2} Suppose that the reset line $v=v_R$ has a finite number $p$ of intersections with $\W^s$, with $w$-coordinates ordered as $w_1<w_2<...<w_l<w_{l+1}<...< w_{l+k}<..<w_p$,  where $l\in \mathbb{N}^*$ denotes the largest index $i$ such that $w_i<\beta$ and exactly $k \geq 2$ intersections lie in $[\beta,\alpha]$: $\beta\leq w_{l+1}, ..., w_{l+k}\leq \alpha$. Let $\mathcal{S}$ denote the set of numbers of small oscillations performed by the trajectories with initial condition in $I_{l+1}=(w_{l+1},w_{l+2})$, $I_{l+2}=(w_{l+2},w_{l+3})$, ..., $I_{l+k-1}=(w_{l+k-1},w_{l+k})$ according to the formula \eqref{nosmall}. Then for any $n\in \mathbb{N}^*$ and any sequence $\{ s_i\}_{i=1}^{N}$ with each $s_i\in \mathcal{S}$, there exists an interval $\mathcal{J}\subset [\beta,\alpha]$ such that every initial condition $w\in\mathcal{J}$ yields a transient MMO with the pattern $1^{s_1} 1^{s_2} ... 1^{s_N}$.

If the number of intersections $w_i$ of the reset line $v=v_R$  with $\W^s$ is infinite but only $k$ of them lie in the interval $[\beta,\alpha]$, then we have exactly two possibilities:
\begin{itemize}
\item all the points $w_i$ in $[\beta,\alpha]$ are not greater than $w^*$ and for every $N\in\mathbb{N}^*$ and every sequence $\{ s_i\}_{i=1}^N$ with $s_i\in \{l+1,l+2,..., l+k-1\}$  there exists an interval $\mathcal{J}\subset [\beta,\alpha]$ such that every initial condition $w\in\mathcal{J}$ yields an MMO with the pattern $1^{s_1} 1^{s_2} ... 1^{s_N}$, where the index $l$ is obtained from the ordering
\[
w_1<w_2<...<w_l<\beta\leq w_{l+1}<...<w_{l+k}\leq \alpha<w_{l+k+1}<...\leq w^*.
\]
\item all the points $w_i$ in $[\beta,\alpha]$ are greater than $w^*$ and for every $N\in\mathbb{N}^*$ and every sequence $\{ s_i\}_{i=1}^N$ with $s_i\in \{l+3/2,l+5/2,..., l+(k-1)+1/2\}$  there exists an interval $\mathcal{J}\subset [\beta,\alpha]$ of initial conditions yielding MMOs with the pattern $1^{s_1} 1^{s_2} ... 1^{s_N}$, where the index $l$ is obtained from the ordering 
\[
w_1>...>w_l>\alpha\geq w_{l+1}>...>w_{l+k}\geq\beta>w_{l+k+1}>...\geq w^*.
\]
\end{itemize}
\end{corollary}

This corollary covers all cases studied in this paper, including finite or infinite number of intersections of the stable manifold with the reset line. Only the number of these points in the  interval $[\beta,\alpha]$ determines the possible MMO patterns.

\section{Adaptation map with one discontinuity point in the invariant interval}\label{sec:OneDisc}
The general description developed above does not yield a precise specification of the dynamics of the system. For  clarity of exposition, from now on, we shall assume that $v_R<v_{-}$ and we focus chiefly on the case where the adaptation map has exactly one discontinuity point $w_1$ in the interval $[\beta,\alpha]$ (although there may be arbitrarily many outside of that interval).  One of the main limitations of this situation is that the resulting MMOs have at most one small oscillation between spikes. Nonetheless,  this case is advantageous in that the number of possible configurations of the map and identity line remains relatively limited, while there is a combinatorial explosion in cases with more intersections.  It will be clear that most of our techniques extend beyond these situations under suitable technical assumptions\footnote{In particular, the unique discontinuity point $w_1$ in $[\beta,\alpha]$ might be replaced by any $w_i$ (with respect to the notation introduced before  \eqref{nosmall}) such that $w_i<w^*$, and satisfying corresponding conditions  (C1)-(C3) below with $w_1$ and $w_2$ replaced, respectively, by $w_i$ and $w_{i+1}$.}. 

The shape of the map depends on the relation of certain points,  
as represented in Fig.~\ref{Partition_d_gamma}, and we list several relevant conditions that we will consider as we proceed:

\begin{description}
	\item[(C1)] There exists a unique discontinuity point $w_1$ in the interval $[\beta,\alpha]$:
	\begin{equation}\label{assump:I}
		\beta < w_1 < \alpha <w_2.
	\end{equation}
	\item[(C2)] The map is piecewise increasing on $[\beta,\alpha]$, i.e.
	\begin{equation}\label{assump:II}
		\alpha<w^*.
	\end{equation}
	\item[(C2')] Alternatively to \textbf{(C2)}, 
	\begin{equation}\label{assump:IIprime}
		w_1<w^*\leq \alpha.
	\end{equation}
	\item[(C3)]  The interval  $\I:=[\beta,\alpha]$ is invariant, which is set by the conditions:
	\begin{equation}\label{assump:III}
		\Phi(\beta)\geq \beta \ \textrm{and} \ \Phi(\alpha)\geq \beta.
	\end{equation}
\end{description}

\begin{figure}[htbp]
\centering
\includegraphics[width=0.92\textwidth]{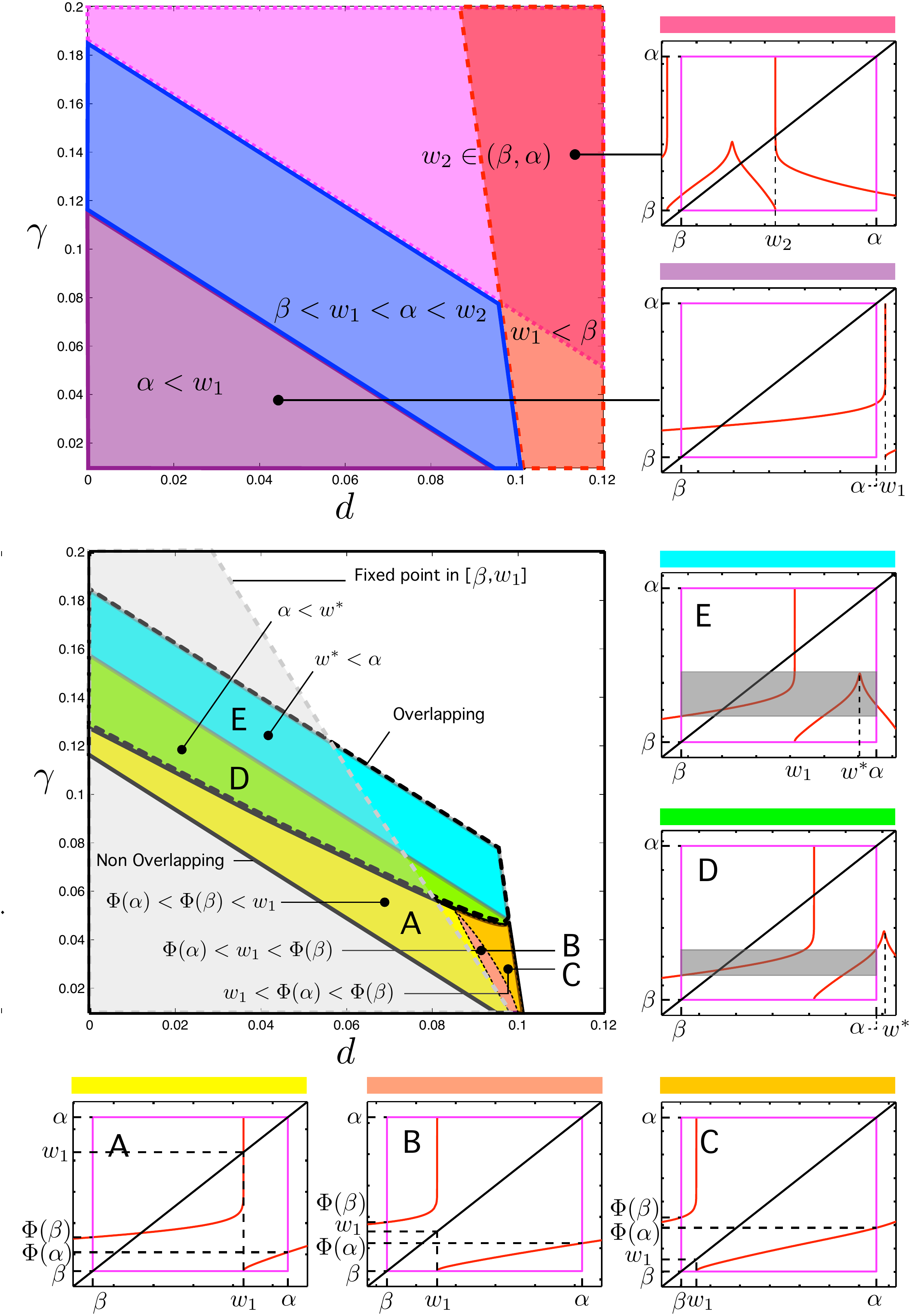}
\caption{Partitions of $(d,\gamma)$ parameter space (for fixed values of the other parameters) according to geometric properties of the map $\Phi$ for the quartic model ($F=v^4+2av$, $a=\eps=0.1$, $b=1$, $I=0.1175$ and $v_R=0.1158$) assuming only two intersections of the reset line with the stable manifold. The blue region in the top partition corresponds to the $(d,\gamma)$ values for which assumption \textbf{(C1)} is satisfied. The bottom partition of that region specifies the subcases of interest: the non-overlapping case (assumption \textbf{(C4)}) is decomposed into regions A (yellow), B (pink) and C (orange) according to the position of the discontinuity point $w_1$ with respect to $\Phi(\alpha)$ and $\Phi(\beta)$. The overlapping case (assumption \textbf{(C4')}) is decomposed into regions D (assumption \textbf{(C2)}) and E (assumption \textbf{(C2')}) according to the position of the critical point $w^*$. Note that, for the given value of $v_R$, in the range of $(d,\gamma)$ values, assumption \textbf{(C3)} is always satisfied. For each subregion, a prototypical scheme of adaptation map $\Phi$ is displayed. The grey regions in the map plots D and E highlight the overlap.}
\label{Partition_d_gamma}
\end{figure}

Non-transient regimes only depend on the properties of the map $\Phi$ in a bounded invariant interval. Indeed, we have seen in Theorem \ref{prop_phi} that $\Phi$ is bounded above and that for $w$ small enough $\Phi(w)>w$, implying the existence of an invariant compact set $\mathcal{I}$ in which any trajectory is trapped after a finite number of iterations. This remark opens the way to consider $\Phi$ as a circle map (after identifying the endpoints of $\mathcal{I}$) and thus to use rotation theory in order to rigorously discriminate (i) whether the firing is regular, bursting or chaotic, corresponding respectively to fixed points, periodic orbits, or chaotic (non-periodic) orbits of $\Phi$ (see~\cite{touboul-brette:09}), as well as (ii) the number of small oscillations occurring before a spike, according to the partition of Fig.~\ref{fig:Structure}, i.e., the signature of the MMO pattern fired.

\begin{definition}\label{def:Lift}
Under assumptions \textbf{(C1)} and \textbf{(C3)}, the invariant interval of $\Phi$ can be defined as $\mathcal{I}=[\beta,\alpha]$. The lift $\Psi$ of $\Phi\vert_{\mathcal{I}}$ is defined for $x\in (\beta,\alpha]$ as:
	\begin{equation}\label{eq:DefLift}
		\Psi:=x \mapsto
		\begin{cases}
			\Phi(x) & \text{if } \beta<x<w_1\\
			\alpha & \text{if } x=w_1\\
			\Phi(x)+(\alpha-\beta) & \text{if } w_1<x\leq \alpha.\\
		\end{cases}
	\end{equation}
	and extended on $\R$ through the relationship that for any $x\in \R$ and $k\in \mathbb{Z}$, 
	\[
	\Psi(x+k(\alpha-\beta))=\Psi(x)+k (\alpha-\beta).
	\]
	The rotation number of $\Psi$ at $w\in\mathbb{R}$ is defined as:
	\begin{equation}\label{def_rot}
		\varrho(\Psi,w):=\lim\limits_{n\to\infty}\frac{\Psi^n(w)-w}{n(\alpha-\beta)}
	\end{equation}
	provided that the limit exists.
\end{definition}

An example of the lift  is given in Fig.~\ref{fig:PsiLPsiR}. Note that the lift is continuous on the interior of the invariant interval $\mathcal{I}$. Indeed, it is continuous on $(\beta,w_{1})$ and $(w_{1},\alpha)$ since $\Phi$ is continuous therein, and at $x=w_{1}$, both its left limit $\Psi(w_{1}^{-})=\Phi(w_{1}^{-})$ and right limit $\Psi(w_{1}^{+})=\Phi(w_{1}^{+})+(\alpha-\beta)$ are equal to $\alpha$. However, the map $\Psi$ is generally not continuous on $\R$ and displays discontinuities at the points $x_k=\alpha+k (\alpha-\beta)$ for $k\in \mathbb{Z}$ when $\Phi(\beta) \neq \Phi(\alpha)$ (which is generally the case). By convention, the above definition introduces $\Psi$ as a left-continuous map. As will be emphasized at relevant places, this choice does not impact our developments, and in particular does not affect possible values of rotation numbers. We finally note that the maps $\Psi$ and $\Phi$ restricted to $[\beta,\alpha]$ induce the same circle map $\varphi:\mathbbm{S}^{\vert\I\vert}\to \mathbbm{S}^{\vert\I\vert}$ on the circle of length ${\vert\I\vert}=\alpha-\beta$, and the orbits of $\Psi$  coincide modulo $ \vert\I\vert$ with the orbits of $\Phi$, except at $w_{1}$ where the map $\Phi$ is not defined\footnote{{The results shown on the orbits of $\Psi$ correspond to actual spike pattern for any initial conditions outside the discrete set of pre-images of $w_{1}$, $\{\Phi^{-n}(w_1), \;n\in \mathbb{N}_0\}$ (where $\mathbb{N}_0:=\{0,1,2,...\}$).}}. Therefore $\Psi$ captures well the general dynamical properties of $\Phi$.

The sign of the jump of $\Psi$ at its discontinuity points $x_k$ will be particularly important in our developments. We will distinguish the following cases:
\begin{description}
	\item[(C4)] We say that the map is \emph{non-overlapping} if \textbf{(C1), (C2)} and \textbf{(C3)} are satisfied, and moreover:
	\begin{equation}\label{assump:IV}
		\Phi(\alpha)\leq \Phi(\beta).
	\end{equation}
\end{description}
 When the   inequality~\eqref{assump:IV} does not hold, we identify another case:
\begin{description}	
\item[(C4')] We say that the map is \emph{overlapping} if conditions \textbf{(C1), (C3)} and either \textbf{(C2)} or \textbf{(C2')} are satisfied and $\Psi$ has a negative jump:
	\begin{equation}\label{assump:IVb}
		\Phi(\alpha)> \Phi(\beta).
			\end{equation}
\end{description}	
	
The terminology follows~\cite{keener} and refers to the property that $\Phi$ is injective in $[\beta,\alpha]$ under assumption \textbf{(C4)}, while the images of $(\beta,w_{1})$ and $(w_{1},\alpha)$ under $\Phi$ have non-empty intersections (\emph{overlap}) under assumptions \textbf{(C4')}. We also emphasize that in the non-overlapping (resp., overlapping) case, the map $\Psi$ has non-negative (resp., negative) jumps at its discontinuity points $(x_k)_{k\in\mathbb{Z}}$, i.e., $\Psi(x_k^{-})\leq \Psi(x_k^+)$ (respectively, $\Psi(x_k^{-})> \Psi(x_k^+)$). 

These conditions may seem complex to check theoretically since they involve relative values for the adaptation map $\Phi$, the discontinuity points, and $\alpha$ and $\beta$. However, they are very easy to check numerically for a specific set of 
parameters. In Fig.~\ref{Partition_d_gamma} we illustrate these different situations for a  quartic model with a particular choice of the subthreshold parameters and for a fixed value of the reset voltage $v_R$, and we identify the regions with respect to the reset parameters $\gamma$ and $d$ where the above conditions are satisfied.

We will  provide an exhaustive description of the MMO patterns produced by the adaptation map when it  has exactly one  discontinuity in the interval $\I$. 
In our framework, we can classify MMOs with half-oscillation precision. However, in this section, we choose for the sake of simplicity in the formulation of the results to consider integer numbers of small oscillations; that is, the points in  $(\beta,w_1)$ correspond to no small oscillations whereas the points in $(w_1,\alpha)$ result in one small oscillation. Thus, referring to the signature of MMOs, we have $s_i=0$ or $s_i=1$ and by grouping together in the signature consecutive spikes followed by no small oscillations, we can assume that $s_i=1$ for any $i$.

We start with a simple remark stating, roughly speaking, that MMOs occur frequently in our system: 
\begin{proposition}\label{forMMBO} Under conditions \textbf{(C1)}, \textbf{(C3)} and either \textbf{(C2)} or \textbf{(C2')}, all orbits of the system \eqref{eq:SubthreshDyn}-\eqref{eq:new_reset} with initial conditions $w\in \mathcal{I}$, except for possible fixed points of $\Phi$ in $[\beta,w_1)$ or orbits attracted by such fixed points, display persistent MMOs. In particular, every periodic orbit of $\Phi$ in $\mathcal{I}$ with period $q>1$ corresponds to regular MMOs (i.e., MMOs with a periodic signature) of the system. 

Under \textbf{(C1)}, \textbf{(C2)} and \textbf{(C3)}, the MMOs displayed by periodic orbits of  $\Phi\vert_{\mathcal{I}}$   are MMBOs. 
\end{proposition}

\begin{proof}
	The first two statements about MMOs follow  from the monotonicity of $\Phi$ in $[\beta,w_1)$ and its limits at $w_1$; indeed, it is easy to see that because of these properties the considered orbits recurrently visit the set $(w_1,\alpha]$, whereas any point of the orbit in this set undergoes one small oscillation before firing a spike. Hence persistent MMOs result, which are regular if these orbits are periodic. 
	
	Similarly under all  assumptions \textbf{(C1)}, \textbf{(C2)} and \textbf{(C3)}, $\Phi$ is monotone increasing on $(w_1,\alpha]$ in addition to $[\beta,w_1)$, with $\Phi(w_1^-)>w_1>\Phi(w_1^+)$. Hence, if an orbit of  $\Phi\vert_{\mathcal{I}}$ is not trapped by a fixed point in one of these intervals, then it necessarily escapes to the other. In particular, any non-trivial periodic orbit thus features small oscillations as well as consecutive spikes with no small oscillations in between, leading to MMBOs.
\end{proof}

Note that under \textbf{(C1)}, \textbf{(C2')}, \textbf{(C3)} it is possible that there are periodic orbits fully contained in $(w_1,\alpha]$. Such periodic orbits always display one small oscillation before each spike. Hence these are MMOs but not MMBOs. 

When conditions \textbf{(C1)}, \textbf{(C2)} and \textbf{(C3)} are satisfied, the singular case $\Phi(\beta)=\Phi(\alpha)$ can be treated using the classical Poincar\'e theory of orientation preserving circle homeomorphisms. In all other cases the corresponding lift is discontinuous and possibly non-monotonic. Our study will build  upon previous works of Keener \cite{keener}, Misiurewicz \cite{misiu}, Rhodes and Thompson \cite{frrhodes,frrhodes2} and  Brette \cite{brette}. We link their general results to MMOs in our system, as well as extend and strengthen some of them to more specific subcases arising in our study, allowing for more refined characterization of the dynamics of $\Phi$.

\subsection{Non-overlapping case}\label{nonoversec}
We start by investigating the non-overlapping case \textbf{(C4)}. In that situation, the lift $\Psi$ is discontinuous (unless $\Phi(\beta)=\Phi(\alpha)$) but conserves the orientation-preserving property since it only admits positive jumps. It is well-known that monotone circle maps conserve the properties of continuous orientation-preserving maps: they have a unique rotation number, and rational rotation numbers imply asymptotically periodic behaviors.

To ensure convergence towards periodic orbits, one needs to take special care about the presence of discontinuities. Indeed, when $\Phi$ has a periodic orbit with period $q$, then necessarily there exists $x_0\in\mathbb{R}$ such that $\Psi^q(x_0)=x_0+p(\alpha-\beta)$ for some $p\in \mathbb{N}^*$, $p, \ q$ relatively prime, i.e. $x_0$ is periodic mod$(\alpha-\beta)$ for the lift $\Psi$. However, since map $\Phi$ is discontinuous at $w_1$, it might happen that, although the rotation number is rational, no truly periodic orbit of $\Phi$ exists but point $w_1$ acts as a periodic point. This means that one of the two following properties is necessarily fulfilled, with $x_0 \; \mbox{mod} \; \vert \I\vert=w_1$ (see \cite{frrhodes}):
\begin{itemize}
\item for all $t\in \mathbb{R}$, $\Psi^q(t)>t+p\vert \I\vert$ and
\begin{equation}
\exists x_0\in\mathbb{R}, \quad \lim\limits_{t\to x_0^-}\Psi^{q}(t)=x_0+p\vert \I\vert;
\end{equation}
\item for all $t\in \mathbb{R}$, $\Psi^q(t)<t+p\vert \I\vert$ and
\begin{equation}
\exists x_0\in\mathbb{R}, \quad \lim\limits_{t\to x_0^+}\Psi^{q}(t)=x_0+p\vert \I\vert.
\end{equation}
\end{itemize}

\begin{remark}
	By allowing the lift to be bivalued at the discontinuity points, Brette~\cite{brette} and Granados \emph{et al}~\cite{alseda} avoid the distinction of the three cases for rational rotation numbers (i.e., the existence of the actual periodic orbit and the two cases listed above). That formalism indeed ensures that the periodic orbit always exists, since the two situations above can happen only if $x_0 \; \mbox{mod} \; \vert \I\vert=w_1$, i.e. when the periodic orbit bifurcates.
\end{remark}

For simplicity, with a little abuse of terminology, in both above cases, we will refer to the orbit of $x_0 \; (\mbox{mod} \vert \I\vert)$ under $\Phi$ as the periodic orbit. Bearing that in mind we now relate the orbits of $\Phi$  to the dynamics of the neuron model and show that the rotation number in the non-overlapping case fully characterizes the signature of the resulting MMO.

\begin{theorem}\label{dynamics2} We assume that the adaptation map $\Phi$ satisfies condition \textbf{(C4)} and consider its lift $\Psi:\R \to \R$. Then the rotation number $\varrho:=\varrho(\Psi,w)$ of $\Psi$ exists and does not depend on $w \in \R$.
	
	Moreover, the rotation number is rational, $\varrho=p/q\in\mathbb{Q}$ with $p\in \N_0:=\{0,1,2,...\}$ and $q\in\N^*$ relatively prime, if and only if $\Phi$ has a periodic orbit, which is related to the MMO pattern fired in the following way:
\begin{enumerate}
\item[(i)] If $\varrho=0$, then the model generates tonic asymptotically regular spiking for every initial condition $w_0\in \lbrack\beta,\alpha\rbrack\setminus\{w_1\}$ (see Figure \ref{RegSpike_MMBO}, top).
\item[(ii)] If $\varrho=1$, then the model generates asymptotically regular MMOs for every initial condition $w_0\in \lbrack\beta,\alpha\rbrack\setminus\{w_1\}$, with  periodic signature $1^11^11^1...=(1^1).$
\item[(iii)] If $\varrho=p/q\in\mathbb{Q}\setminus\mathbb{Z}$ ($p, q$ relatively prime, $q>1$ and $1\leq p<q$), then the model generates asymptotically regular MMBOs for every initial condition $w_0\in \lbrack\beta,\alpha\rbrack\setminus\{w_1\}$ (see e.g., Figure \ref{RegSpike_MMBO}, bottom). Defining $0< l_1<\cdots <l_p\leq q-1$ as the unique integers such that $l_{i}p/q \mod 1 \geq (q-p)/q$ and $\mathcal{L}_i=l_{i+1}-l_i$ for $i=1\cdots p$ (with the convention $l_{p+1}=q+1$), the MMBO signature is $\mathcal{L}_1^1\cdots \mathcal{L}_p^1$.
\item[(iv)] If $\varrho\in\mathbb{R}\setminus\mathbb{Q}$, then there are no fixed point and no periodic orbit, and the system fires chaotic MMOs.
\end{enumerate}
\end{theorem}

\begin{figure}[htbp]
\centering
\includegraphics[width=0.8\textwidth]{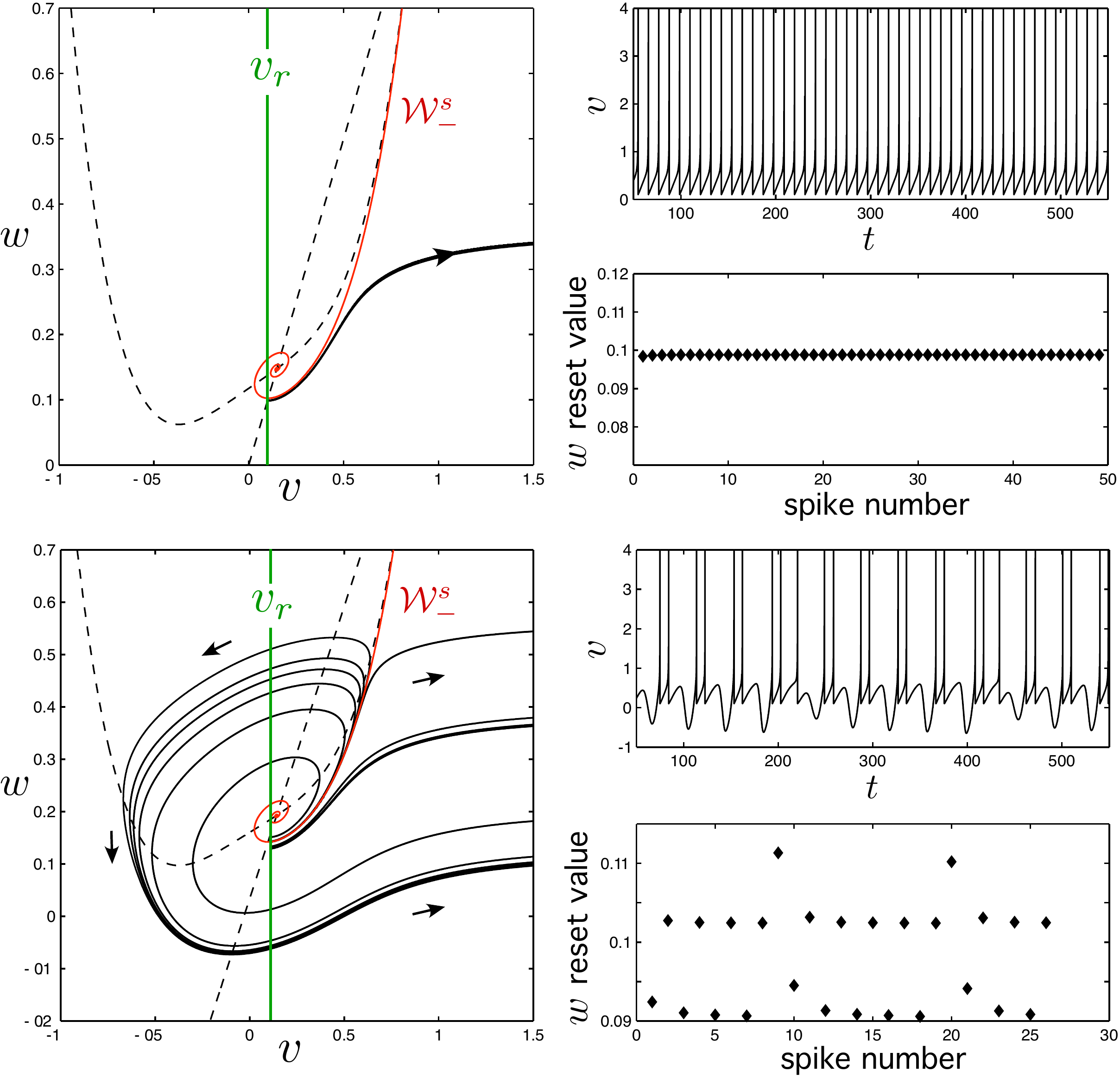}
\caption{Phase plane structure, $v$ signal generated along attractive periodic orbits and sequence of $w$ reset values for two sets of parameter values for which the map $\Phi$ is in the non-overlapping case \textbf{(C4)}. In both cases, $v_R=0.1$ and $\gamma=0.05$. The top case ($d=0.08$) illustrates the regular spiking behavior corresponding to the rotation number $\varrho = 0$. The bottom case ($d=0.08657$) displays a complex MMBO periodic orbit with associated rational rotation number.}
\label{RegSpike_MMBO}
\end{figure}

\begin{remark}
This result establishes that in the non-overlapping case, the MMO signatures of orbits are determined by the rotation number and provides a constructive algorithm to compute the MMO signature associated to a given rotation number. We illustrate this construction on two examples:
	\begin{itemize}
		\item Orbits of the adaptation map with rotation number $\varrho=1/q$ have signature $q^1$. When $\varrho=(q-1)/q$ the signature is $2^11^1\cdots 1^1$ (with $q-2$ repetitions of the pattern $1^1$).
		\item For the rotation number $\varrho=3/5$, up to cyclic ordering, the periodic orbits are ordered as those of the corresponding rotation by $3/5$ on the unit circle, i.e. $\{0, \frac{3}{5}, \frac{2\cdot3}{5}= \frac{1}{5} \mod 1, \frac{3\cdot3}{5}= \frac{4}{5}\mod 1, \frac{4\cdot3}{5}= \frac{2}{5} \mod 1\}$. The three indices corresponding to values greater  or equal to $2/5$ are $\{1,3,4\}$; hence, $\mathcal{L}_1=3-1=2$, $\mathcal{L}_2=4-3=1$, $\mathcal{L}_3=1+5-4=2$, and the signature is $2^11^12^1$.
		
	\end{itemize}
\end{remark}

\begin{proof} Since the induced lift $\Psi:\mathbb{R}\to\mathbb{R}$ is strictly increasing, we can apply the theory of monotone circle maps theory developed by Rhodes and Thompson~\cite{frrhodes,frrhodes2} and Brette~\cite{brette}. The existence and uniqueness of the rotation number is shown in~\cite[Theorem 1]{frrhodes} and~\cite{brette}, and the proof for orientation preserving homeomorphisms applies\footnote{Continuity of the lift is not used in the classical proof of the uniqueness of the rotation number for orientation preserving circle homeomorphisms, see e.g.~\cite[Proposition 11.1.1]{katok}.}. The characterization of the orbits in the case of rational rotation numbers results from~\cite[Theorem 2]{frrhodes} and the fact that $\Psi$ is strictly increasing.

Moreover, if $\varrho=p/q$, then it can be shown that every non-periodic point $w\in [\beta,\alpha]$ of $\Phi$ tends under $\Phi^q$ to some periodic point $\tilde{w}\in [\beta,\alpha]$: $\lim_{n\to\infty}\Phi^{nq}(w)=\tilde{w}$. This is a consequence of~\cite[Proposition 5]{brette} since the monotonicity of $\Psi$ ensures that the underlying circle map is strictly orientation preserving. From the proof therein it also follows that the asymptotic behavior is consistent for all the points of a given orbit, i.e. that if $w$ tends under $\Phi^q$ to $\tilde{w}$, then $\Phi^k(w)$, $k=0,1,..., q-1$, tends to its corresponding point $\Phi^k(\tilde{w})$ on the periodic orbit of $\tilde{w}$. This provides the classification of orbits for the adaptation map, analogous to the one for a circle homeomorphism with rational rotation number (cf.~\cite[Proposition 11.2.2]{katok}).  Next, we consider the subcases of firing patterns.

	{\it (i-ii)} When $\varrho(\Psi)=0 \mod 1$, the adaptation map admits a fixed point. Moreover, under the current assumptions and the way we have defined the lift $\Psi$ we either have $\varrho(\Psi)=0$ if the fixed point belongs to $(\beta,w_1)$, in which case there is no (full) small oscillation between spikes, or $\varrho(\Psi)=1$ if the fixed point belongs to $(w_1,\alpha)$, in which case the orbit displays one small oscillation between every two consecutive spikes.

	{\it (iii)} As mentioned in Proposition~\ref{forMMBO}, periodic orbits necessarily correspond to MMBO. Moreover, it is not hard to show that $q$-periodic orbits with rotation numbers $p/q$ have exactly $p$ points to the right of $w_1$. These points split the periodic orbit into firing events consisting of either one spike or a burst, separated by a small oscillation. Since the lift preserves the orientation, the consecutive points of a periodic orbit $\{\bar{w}, \Phi(\bar{w}),\ldots, \Phi^{q-1}(\bar{w})\}$ with rotation number $p/q$ are ordered as the sequence of numbers $(0, p/q, 2p/q, ..., (q-1) /q)$ in $[0,1]$  (up to  cyclic permutation, see e.g.~\cite[Proposition 11.2.1]{katok}). The signature of the MMBO is directly related to the indexes $l\in \{0,1,...,q-1\}$ such that $\Phi^l(\bar{w})>w_1$, and hence such that $l p/q \geq (q-p)/q\mod 1$. We easily conclude that the signature of the MMBO indeed is $\mathcal{L}_1^1\mathcal{L}_2^1\cdots \mathcal{L}_p^1$.
	
	{\it (iv)} If the  rotation number is irrational, then $\Phi$ admits no periodic orbit, and all orbits under $\Phi$ have the same limit set $\Omega$, which is either the circle or a Cantor-type set as in the continuous case ($\Phi(\beta)=\Phi(\alpha)$), as proved in \cite[Proposition 6]{brette}.
\end{proof}

When $w_1$ is periodic mod $(\alpha-\beta)$, the corresponding forward attracting periodic orbit is unique. Otherwise, several attracting periodic orbits may exist with the same rational rotation number, and hence with the same period and the same ordering.  In \cite{alseda}, the authors have proved the uniqueness of the periodic orbit of maps such as $\Phi$ in the non-overlapping case with the assumption that $\Phi$ is contractive on both $(\beta,w_1)$ and $(w_1,\alpha)$. Here, because of the divergence of the differential at the discontinuity points, we cannot use the contraction assumption.

We emphasize that since $\Psi$ is a strictly increasing lift of a degree-one circle map, changing its value at a discontinuity point (while conserving monotonicity) does not change the value of the rotation number (see e.g.,~\cite{frrhodes}). The above remark means that for the characterization of the dynamics of $\Phi$, it does not matter whether we define the lift $\Psi$ to be left- or right-continous at its discontinuity points $\beta+k(\alpha-\beta)$, nor that  $\Phi$ is formally not defined at $w_1$, since $\lim\limits_{w\to w_1^-}\Phi(w)=\alpha$ and $\lim\limits_{w\to w_1^+}\Phi(w)=\beta$.

We now provide a simple sufficient condition for the existence of $2^1$ MMBOs.
This result is analogous to~\cite[Lemma 3.2]{keener} but does not necessitate the boundedness assumption on the derivative of the map made in \cite{keener}, which  our map $\Phi$ obviously does not satisfy.

\begin{proposition}\label{for2per}
Assume that $\Phi$ fulfills condition \textbf{(C4)}  and moreover that $\Phi(\alpha)<w_1<\Phi(\beta)$. Then $\Phi$ admits a periodic orbit of period 2, thus the system has a $2^1$ MMBO.
\end{proposition}

\begin{figure}[htbp]
\centering
\includegraphics[width=0.6\textwidth]{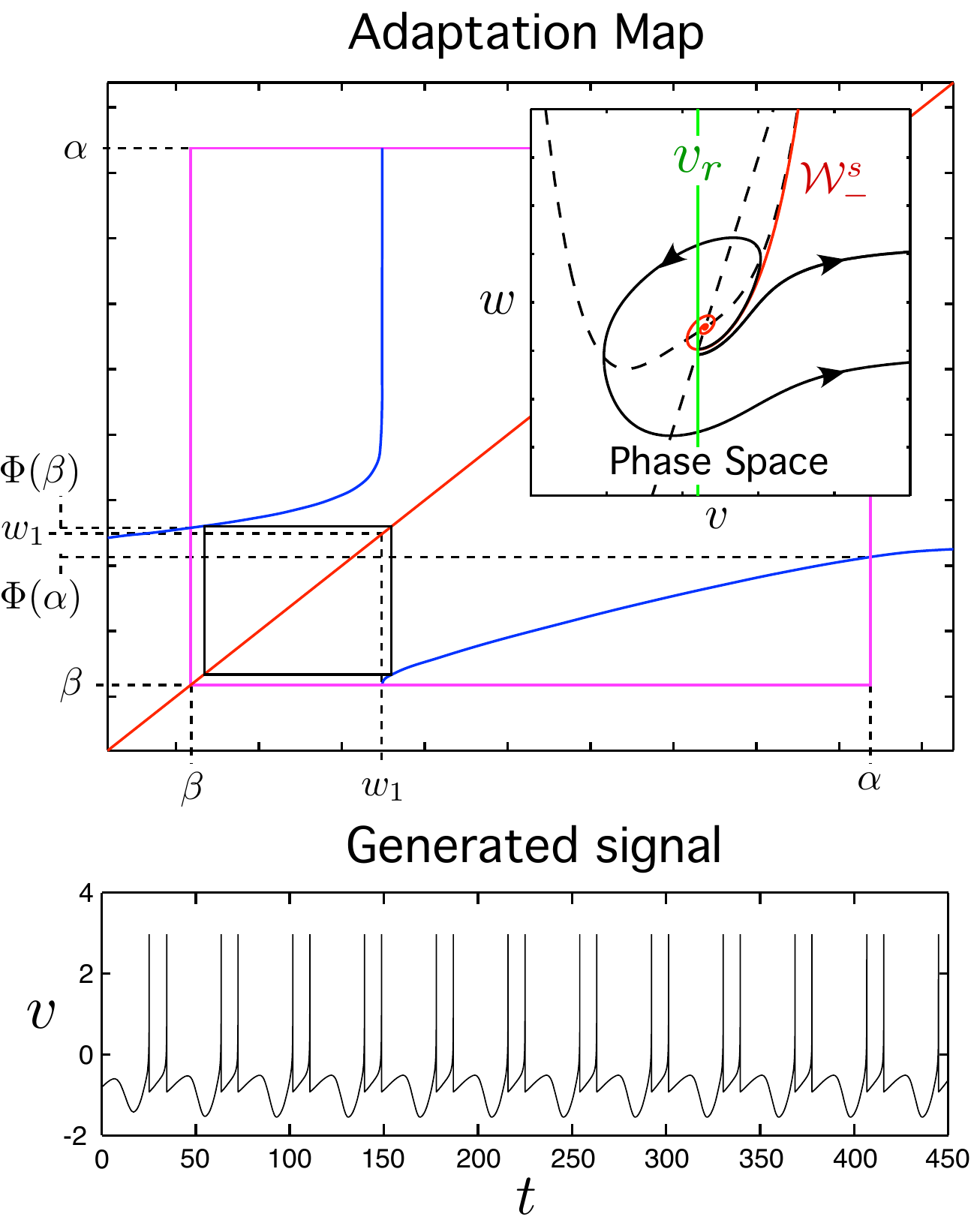}
\caption{Phase plane (inset) and adaptation map (top) fulfilling condition \textbf{(C4)} and the additional condition $\Phi(\alpha)<w_1<\Phi(\beta)$, along with the associated MMBO orbit of system \eqref{eq:SubthreshDyn} (bottom). The rotation number is equal to $0.5$, hence the $v$ signal along the orbit is a periodic alternation of a pair of spikes and one small oscillation. The parameter values of the system  corresponding to this simulation are $v_R=0.1$, $\gamma=0.05$ and $d=0.087$.}
\label{RhoHalf}
\end{figure}

\begin{proof}
	In this case, $\Phi^2((\beta,w_1))\subset (\beta,w_1)$, $\Phi^2$ is continuous on $(\beta,w_1)\cup (w_1,\alpha)$, $\Phi^2(\beta)>\beta$ and $\lim\limits_{w\to w_1^-}\Phi^2(w)=\Phi(\alpha)<w_1$. Hence, $\Phi^2$ admits a fixed point in $(\beta,w_1)$ corresponding to a periodic point of period 2 for $\Phi$. On the other hand, the second point of this periodic orbit lies in $(w_1,\alpha)$ since $\Phi((\beta,w_1))\subset (w_1,\alpha)$. Thus this orbit exhibits MMBO and necessarily $\varrho(\Psi)=1/2$. We illustrate this result in Figure \ref{RhoHalf}.
\end{proof}

	 When the second assumption of Proposition~\ref{for2per} is not valid and $\Phi(\beta)>\Phi(\alpha)>w_1$, the dynamics may generate complex orbits of higher period or even chaos. Different MMBO patterns may therefore be observed in the non-overlapping case,  depending sensitively  on the parameters. We now focus on this dependence on the reset parameters $(d,\gamma)$ and show that the rotation number varies as a devil's staircase (in the sense of Theorem \ref{new_devil} below). This result is based on a theorem in \cite{brette}. However it does not follow from \cite{brette} immediately, since varying reset parameters changes the invariant interval $[\beta,\alpha]$ and one needs to add some technical assumptions to ensure that the lifts display an increasing relation. The detailed proof can be found in the Appendix.
\begin{theorem}\label{new_devil}
Assume that for any $d\in [d_1,d_2]$, the adaptation map $\Phi_d$ remains in the non-overlapping case \textbf{(C4)} and $\Phi_d(\alpha_{d_2})<\Phi_d(\beta_{d_1})$. Let $\varrho_d$ be the unique rotation number of $\Phi_d$. Then:
\begin{itemize}
\item $\rho: d\mapsto\varrho_d$ is continuous and non-decreasing on $[d_1,d_2]$; 
\item for all $p/q\in\mathbb{Q}\cap \textrm{image}(\rho)$, $\rho^{ - 1}(p/q)$ is an interval containing more than one point   except, possibly, at the boundaries of the interval $[d_1,d_2]$;
\item for every irrational $\varrho_d\in \textrm{image}(\rho)$, $\rho^{ - 1}(\varrho_d)$ is a one-point set;
\item the set of points $d$ at which $\rho$ takes irrational values is, up to a countable number of points, a Cantor-type subset of $\lbrack d_1,d_2\rbrack $. 
\end{itemize}
\end{theorem}

A similar result holds for the dependence of the rotation number on the parameter $\gamma$ in the regime where we can ensure the suitable monotonicity of $\gamma\mapsto \varrho_{\gamma}$.  Fig. \ref{fig_Devil_Staircase} illustrates a case where this theorem applies.

\begin{figure}[htbp]
\centering
\includegraphics[width=0.8\textwidth]{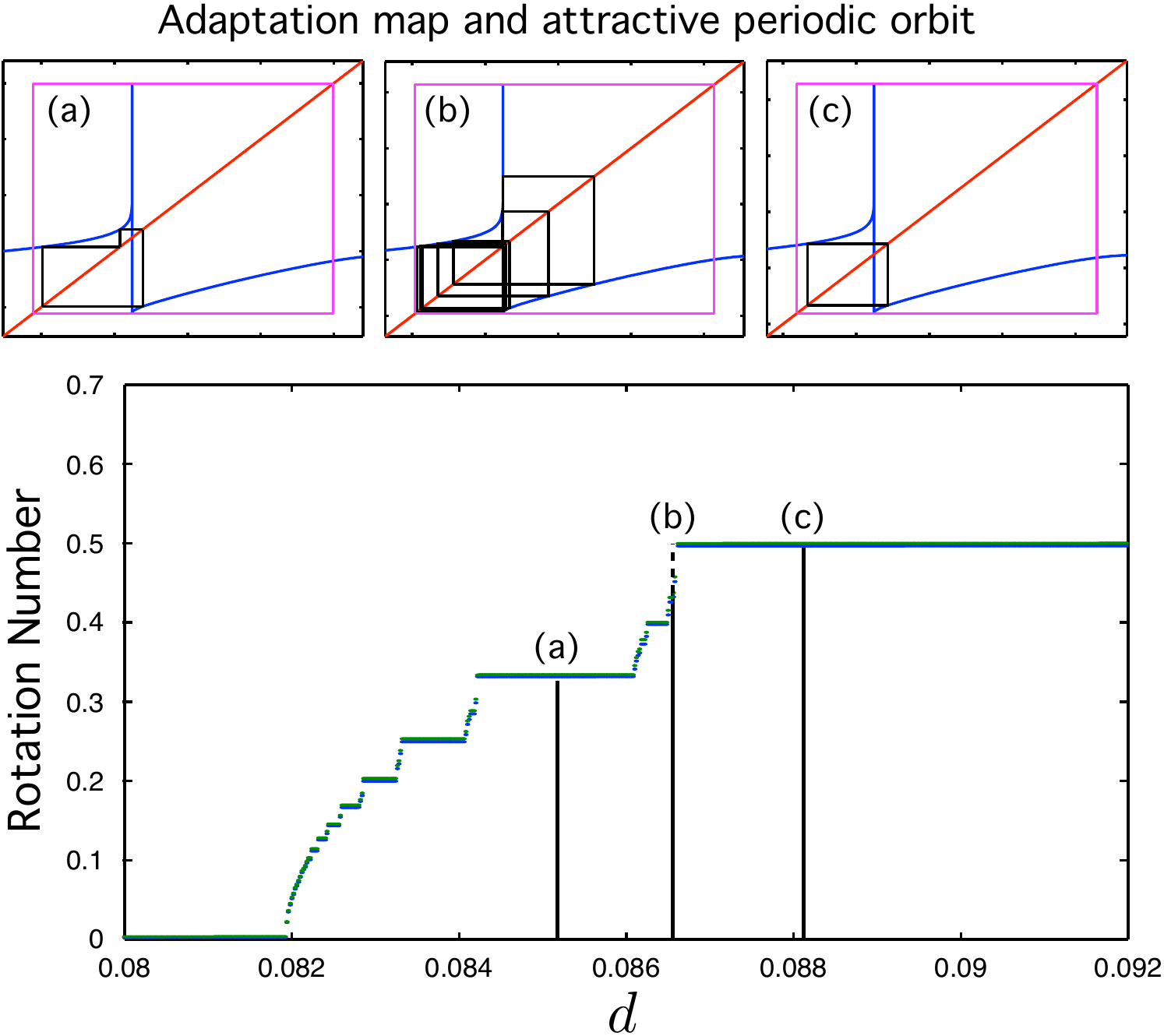}
\caption{Rotation number as a function of $d$. The parameter values $v_R=0.1$ and $\gamma=0.05$ have been chosen such that the adaptive map $\Phi$ fulfills condition \textbf{(C4)} for any value of $d \in [0.08,0.092]$. Theorem~\ref{new_devil} applies here, and the rotation number varies as a devil's staircase, as shown in the bottom plot. The top panels show the adaptation map and corresponding attractive periodic orbit at the $d$ values labelled correspondingly in the rotation number plot; note that the rotation number for case (b) is a rational number between 1/3 and 1/2.} 
\label{fig_Devil_Staircase}
\end{figure}

\subsection{Overlapping case}\label{sec:overlap}
Now let $\Phi$  satisfy the properties of the overlapping case \textbf{(C4')}. In this case, the lift $\Psi$ is no longer increasing (it has negative jumps at the points $x_k=\alpha+k (\alpha-\beta)$ for $k\in \mathbbm{Z}$), and a number of important properties inherited from the well-behaved dynamics of orientation-preserving circle homeomorphisms that persist in the non-overlapping case~\cite{brette,frrhodes,frrhodes2} are now lost, leaving room for still richer dynamics.

In the overlapping case, it is easy to see that our map restricted to its invariant interval $[\beta,\alpha]$ falls in the framework of the so-called \emph{old heavy maps}~\cite{misiu}, since it is a lift of a degree-one circle map with only negative jumps. These maps have interesting dynamics with non-unique rotation numbers. More precisely, we can define a rotation interval $[a(\Psi),b(\Psi)]$ with
\begin{eqnarray}\label{eq:def_ab}
	a(\Psi)&:=\inf_{w\in\mathbb{R}}\liminf_{n\to\infty}\frac{\Psi^n(w)-w}{n(\alpha-\beta)},\\
	b(\Psi)&:=\sup_{w\in\mathbb{R}}\limsup_{n\to\infty}\frac{\Psi^n(w)-w}{n(\alpha-\beta)}.
\end{eqnarray}
As noted in~\cite{misiu}, these two quantities are the (unique) rotation numbers of the continuous orientation preserving maps:
\begin{eqnarray}
\Psi_l(w)&:=& \inf\{\Psi(z): z\geq w\}, \label{a_new}\\
\Psi_r(w)&:=& \sup\{\Psi(z): z\leq w\}; \label{b_new}
\end{eqnarray}
that is, $a(\Psi)=\varrho(\Psi_l)$ and $b(\Psi)=\varrho(\Psi_r)$. The corresponding maps $\Psi_l$ and $\Psi_r$ for the adaptation map of the hybrid neuron model are plotted  in Fig.~\ref{fig:PsiLPsiR}.

\begin{figure}[h]
	\centering
		\includegraphics[width=.7\textwidth]{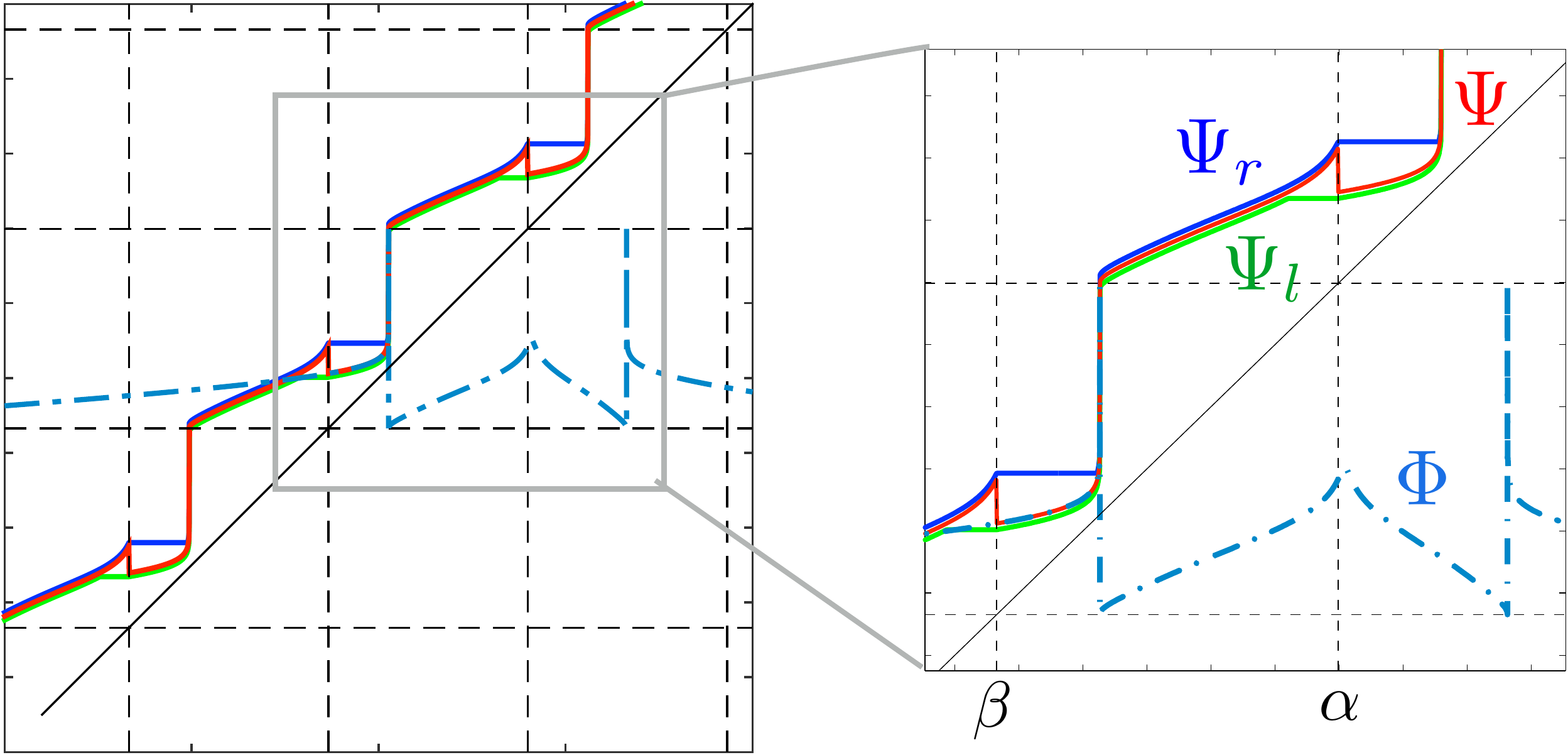}
	\caption{The orientation-preserving maps $\Psi_l$ (green) and $\Psi_r$ (blue) enveloping the lift $\Psi$ (red line), which is non-monotonic and admits negative jumps, for the adaptation map $\Phi$ (blue dashed curve) in the overlapping case.}
	\label{fig:PsiLPsiR}
\end{figure}

 We can now conclude after \cite{misiu}:
 \begin{proposition}\label{Effective}  Under  assumption  \textbf{(C4')}, 
\begin{enumerate}
 \item if $\Phi$ admits a $q$-periodic point $  w$ with rotation number $\varrho(\Psi,w)=p/q$, then $a(\Psi)\leq p/q\leq b(\Psi)$;
  \item if $a(\Psi)<p/q<b(\Psi)$, then $\Phi$ admits a periodic point $w$ of period $q$ and rotation number $\varrho(\Psi,w)=p/q$. 
  \end{enumerate}
In both cases, the orbit of $w$ displays MMOs, unless $p=0$, in which case $w$ is a fixed point in $[\beta,w_1)$. 
Moreover, for any $\varrho_1$ and $\varrho_2$ such that $a(\Psi)\leq \varrho_1\leq \varrho_2\leq b(\Psi)$, there exist $w$ such that
\begin{eqnarray}
	\liminf_{n\to\infty}\frac{\Psi^n(w)-w}{n(\alpha-\beta)}&= \varrho_1, \\
	\limsup_{n\to\infty}\frac{\Psi^n(w)-w}{n(\alpha-\beta)}&= \varrho_2.
\end{eqnarray}
\end{proposition}


This result implies in particular that the rotation set in the overlapping case is closed, meaning that every number $\varrho\in [a(\Psi),b(\Psi)]$ is the rotation number $\varrho(\Psi,w)$  of an orbit with initial condition $w\in [\beta,\alpha]$, and the rational numbers in its interior correspond inevitably to periodic orbits. The property of having a non-trivial rotation interval implies coexistence of infinitely many periodic orbits of distinct periods; this situation is sometimes referred to as `chaos' (see \cite{keener}), although this notion differs from the chaos associated with non-regular (`chaotic') behavior of orbits with non-rational rotation numbers. We next consider (i) the variation of the rotation interval as a function of the reset parameters and (ii) the relationship between rotation intervals and MMO patterns.

A  result from~\cite[Theorem B]{misiu} ensures continuous dependence of the boundaries of the rotation interval $a(\Psi)$ and $b(\Psi)$ on map parameters when $\Psi_l$ and $\Psi_r$, regarded as elements of the space $C^0(\mathbb{R})$ occupied with the uniform topology, depend continuously on these parameters. The following proposition makes this dependence more precise in our case by showing that these vary as a devil's staircase under mild assumptions.

\begin{proposition}\label{prop:overlap_devil} Consider fixed parameters $v_{R}$, $a$, $b$, $\gamma$ and $I$ and vary $d\in \lbrack\lambda_{1},\lambda_{2}\rbrack$ such that, for each $d\in \lbrack\lambda_{1},\lambda_{2}\rbrack$, the corresponding adaptation map $\Phi_d$ satisfies the assumptions of the overlapping case \textbf{(C4')}. Then the maps $d\mapsto a(\Psi_d)$ and $d\mapsto b(\Psi_d)$ assigning to $d$ the endpoints of the rotation interval of $\Phi_d$ are continuous.

If we further assume that, for any pair $(d_1, d_2)\in [\lambda_1,\lambda_2]^2$ with $d_2<d_1$, we have
\begin{equation}\label{assump_overlap}
\Phi_{d_2}(\beta_{d_1})\leq \Phi_{d_2}(\beta_{d_2})+d_1-d_2,
\end{equation}
then the maps $d\mapsto \Psi_d(w)$, $d\mapsto \Psi_{d,r}(w)$ and $d\mapsto \Psi_{d,l}(w)$ are increasing for each $w$ ($\Psi_{d,r}$ and $\Psi_{d,l}$ denote, respectively,  upper and lower enveloping maps of the lift $\Psi_d$ of $\Phi_d$). Consequently, the maps $d\mapsto a(\Psi_d)$ and $d\mapsto b(\Psi_d)$ behave like a devil's staircase.
\end{proposition}

We note that the sufficient condition (\ref{assump_overlap}) is equivalent to
\begin{equation}\label{assump_equiv}
\Psi_{d_2}(\beta_{d_1})\leq \Psi_{d_2}(\beta_{d_2}^{+})+d_1-d_2,
\end{equation}
where $\Psi_{d_2}(\beta_{d_2}^{+})$ denotes the right limit of $\Psi_{d_2}$ at $\beta_{d_2}$. This latter condition is satisfied for instance when, for every $d\in [\lambda_1,\lambda_2]$, $\Phi^{\prime}_d<1$ in the whole interval $[\beta_d,\beta_d+(\lambda_2-\lambda_1)]$.
The proposition is proved in Appendix~\ref{sec:appendix} and illustrated in Fig. \ref{Dev_St_Psir_Psil_Phi}.

 \begin{figure}[htbp]
 \centering
\includegraphics[width=0.7\textwidth]{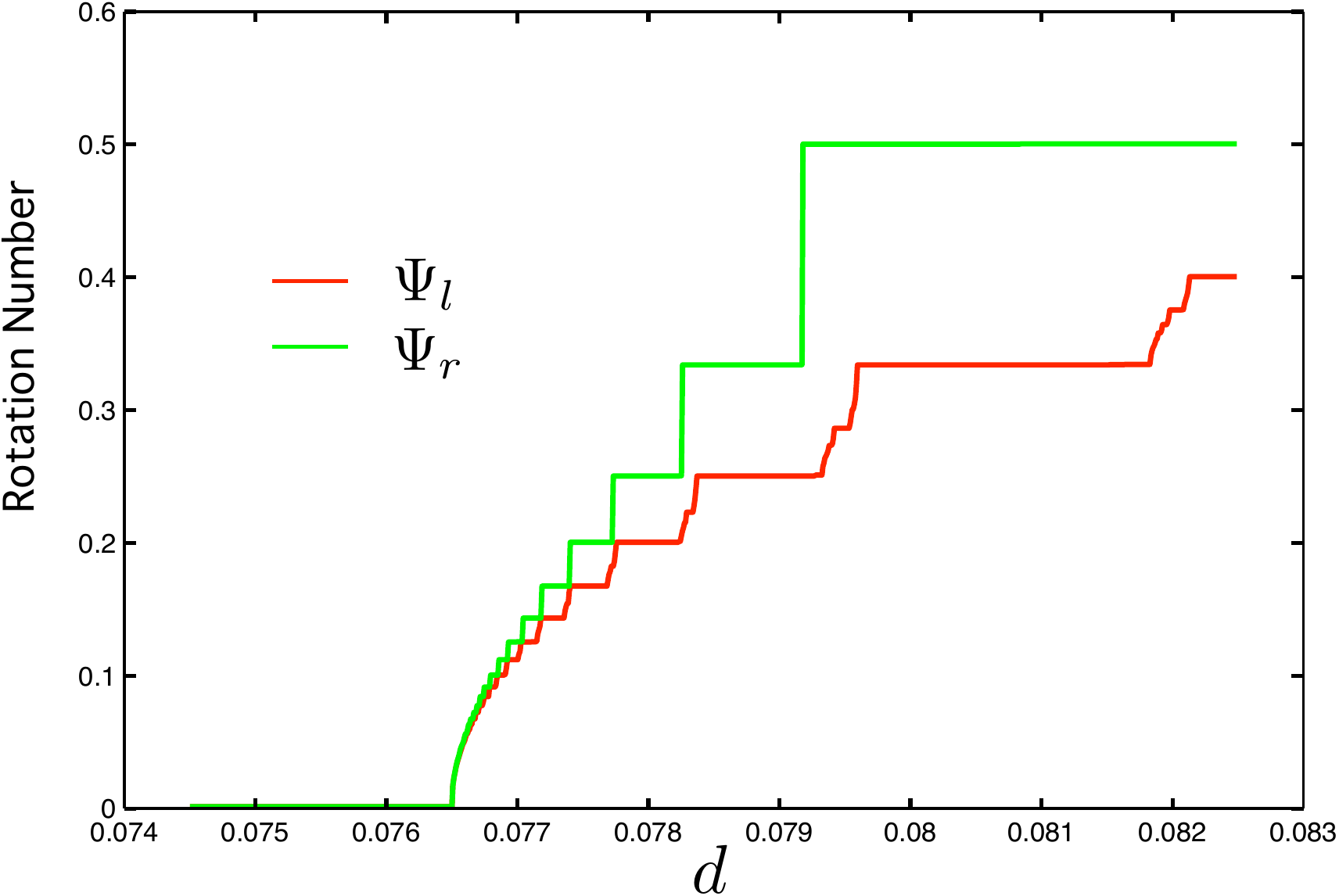}
\caption{Rotation intervals for the lifts $\Psi_{d,l}, \Psi_{d,r}$ of the adaptation maps $\Phi_d$ for a range of $d$. The parameter value $\gamma=0.05$ has been chosen so that $\Phi_d$ remains in the overlapping case for all $d \in [0.0745, 0.0825]$.}
\label{Dev_St_Psir_Psil_Phi}
\end{figure}

In the overlapping case, the pointwise rotation numbers $\varrho(\Psi,w)$ may not exist for some initial conditions and generally depend upon $w$, implying that we have non-trivial (i.e., non-singleton) rotation intervals. {Moreover, despite Proposition \ref{Effective}, even knowing the rotation interval $[a(\Psi),\beta(\Psi)$] does not fully determine yet the structure of the set of all (minimal) periods of orbits of $\Phi$.} Specific cases were fully characterized, however, notably degree-one continuous non-injective circle maps~\cite{llibre,misiurot}. For maps with discontinuities, the issue is very complex and, to our knowledge, periods of periodic orbits are completely described only for lifts of monotonic modulo 1 transformations (see \cite{hofbauer}), which corresponds to the overlapping case with the additional monotonicity assumption \textbf{(C2)}. In addition to these difficulties, the overlap prevents systematic deduction of the MMO signature from knowledge of the rotation number, since the rotation number $\varrho(\Psi,w)$ does not determine the ordering of points on the orbit of $w$.

Specific analysis on the maps considered here, however, provides some information about this characterization. First, when \textbf{(C2)} holds, the map $\Phi$ is piecewise increasing and thus Proposition~\ref{forMMBO} applies and ensures that periodic orbits are associated to MMBOs. When \textbf{(C2)} is not valid, we still know that periodic orbits fully contained in $(w_1,\alpha)$ correspond to MMOs with signature $1^1$. Beyond these particular cases, we now demonstrate more general results under milder assumptions.

 \begin{proposition}\label{male1} Assume that $\Phi$ fulfills \textbf{(C4')} and admits at least two fixed points, $w_f \in \lbrack \beta, w_1)$ and $\hat{w}_f\in (w_1,\alpha)$. Then there exist periodic orbits of arbitrary period displaying MMOs (which are MMBOs under \textbf{(C2)}).   
 \end{proposition}

\begin{proof}
	First of all, we note that since $w_{f}\in \lbrack \beta, w_1)$ is a fixed point of $\Phi$, it is also a fixed point for $\Psi$ and thus the associated rotation number is equal to $0$. Moreover, for $\hat{w}_f\in (w_1,\alpha)$  we have $\Psi(\hat{w}_f)=\hat{w}_f+(\alpha-\beta)$  and thus the associated rotation number under $\Psi$ is equal to $1$. We thus conclude that the rotation interval of $\Psi$ contains the full interval $\lbrack0,1\rbrack$, which concludes the proof.
	\end{proof}

 \begin{proposition}\label{male2} Assume that $\Phi$ satisfies \textbf{(C4')} and that $\Phi$ admits at least one fixed point in $(w_1,\alpha)$, the smallest of which we denote by $w_f\in (w_1,\alpha)$. Assume moreover that there is no fixed point of $\Phi$ in $\lbrack\beta,w_1)$. Then
 \begin{itemize}
   \item if $\Phi(\beta)<w_f$, then there exists $\tilde{q}>1$ such that for all $q>\tilde{q}$, $\Phi$ admits a periodic orbit of period $q$, and the associated trajectories display MMOs;
   \item if $\Phi(\beta)\geq w_f$, then $\Phi$ admits a trivial rotation interval $\lbrack a(\Psi),b(\Psi)\rbrack=\{1\}$ and periodic orbits correspond to MMOs with signature $1^1$. If additionally $\alpha\leq w^*$, then $\Phi$ admits no periodic orbit of period $q>1$, every orbit converges towards a fixed point in $(w_1,\alpha]$, and associated trajectories display asymptotically regular MMOs with signature $1^1$. 
 \end{itemize}
 \end{proposition}

\begin{proof}
	We first assume that $\Phi(\beta)<w_f$. In this case, the lower envelope $\Psi_l$ intersects neither the identity ($\textrm{Id}$) line nor the $\textrm{Id}+(\alpha-\beta)$ line (and, obviously, none of the lines $\textrm{Id}+k(\alpha-\beta)$ for $k\in\mathbb{Z}$). Thus the graph of $\Psi_l$ is fully contained between the lines $\textrm{Id}$ and $\textrm{Id}+(\alpha-\beta)$ and since the functions $\Psi_l(w)-w+k(\alpha-\beta)$, $k\in\mathbb{Z}$, are continuous and $\alpha-\beta$ periodic, there exists $\delta>0$ such that for every $w$ and $n\in \N^*$ we have $\Psi^n_l(w)<w+n(\alpha-\beta)-n\delta$ and
	\[
	a(\Psi)=\varrho(\Psi_l)<1-\frac{\delta}{(\alpha-\beta)}.
	\]
	On the other hand, $\Psi_r(w_f)=w_f+(\alpha-\beta)$ and thus $b(\Psi)=\varrho(\Psi_r)=1$. Therefore, the rotation interval is not trivial and
	\[
	\lbrack 1-\frac{\delta}{(\alpha-\beta)},1\rbrack\subset \lbrack a(\Psi), b(\Psi)\rbrack.
	\]
	For every $q>1$ large enough, we have
	\[
	a(\Psi)<\frac{q-1}{q}<b(\Psi);
	\]
	that is,  there exists a periodic orbit of $\Phi$ with period $q$ and rotation number $\frac{q-1}{q}$.

	We now assume $\Phi(\beta)\geq w_f$. Then $w_f$ is also a fixed point mod $(\alpha-\beta)$ of $\Psi_l$, i.e.
	 \begin{eqnarray*}
	 & \Psi_l(w_f) = w_f+(\alpha-\beta), \\
	 & a(\Psi)=\varrho(\Psi_l)=\varrho(\Psi_r)=b(\Psi)=1.
	 \end{eqnarray*}
	 Thus, if there was some periodic orbit of period $q>1$, all points of such an orbit would lie in $(w_1,\alpha)$ and would have rotation number $1$, yielding MMOs with signature $1^1$. However, if additionally $\alpha\leq w^*$, then the map $\Phi$ is increasing in $(w_1,\alpha)$ and no periodic orbit can be fully contained in this interval. 
	 
	Assuming that $\alpha\leq w^*$, we notice that the interval $[w_f,\alpha]$ is invariant for $\Phi$ and that $\Phi$ is continuous and increasing therein. Consequently, every point $w\in [w_f,\alpha]$ tends under $\Phi$ to one of the fixed points in $[w_f,\alpha]$. But as every point in $[\beta,w_f)$ in mapped finally into $[w_f,\alpha]$, this holds for all the points in $[\beta,\alpha]$ and the proof is completed.
\end{proof}

Later, we shall complement the above result in a slightly more general situation, in Theorem \ref{unique_fixed_point},  by treating  maps admitting fixed points in $\lbrack \beta,w_1)$ and lacking a fixed point in $(w_1,\alpha\rbrack$. 
We can also easily justify the following:
\begin{corollary} In the overlapping case, the existence of a fixed point of $\Phi$  and of a periodic orbit with rotation number $p/q$, for some period $q>1$ and $p\neq q$, implies the existence of periodic orbits with all arbitrary periods greater than $q$, each yielding MMOs. 

In particular, if there exist a fixed point and a periodic orbit of rotation number $1/2$, then there are periodic orbits of all periods exhibiting MMOs. Similarly, as already proved, if there are a fixed point in $(\beta,w_1)$ and a fixed point in $(w_1,\alpha)$, then $[a(\Psi),b(\Psi)]=[0,1]$ and there are periodic orbits of all periods, with MMOs.
\end{corollary}

In contrast to the non-overlapping case, in the overlapping case we have dropped the assumption \textbf{(C2)}  that the map is piecewise increasing. However, under this assumption we can describe the chaotic behavior of the map's iterates more precisely: 	

\begin{corollary}\label{shift} Assume that $\Phi$ satisfies \textbf{(C4')} with \textbf{(C2)}, that $\Phi(\alpha)<w_1$, that $\Phi$ has at least two periodic orbits with periods $q_1\neq q_2$ and that exactly one point of each of these periodic orbits is greater than $w_1$. Then  the mapping $w\mapsto \Phi(w)$ is a shift on a sequence space.
\end{corollary}

To obtain the above proposition it suffices, for example, to look at the proof of \cite[Theorem 2.4]{keener} and notice that the piecewise contraction assumption made in \cite{keener} does not interfere in the proof of this particular theorem and therefore it extends to our class of discontinuous maps with unbounded derivative. 

\subsection{A general result for adaptation maps with one discontinuity in the invariant interval $[\beta,\alpha]$}\label{section:General}
In previous sections we have classified the dynamics of the adaptation map and the associated spiking patterns in terms of rotation numbers and rotation intervals. However, for particular values of the parameters, we lack an explicit analytical expression with which to characterize the corresponding rotation properties.   Below we address this problem under the assumptions \textbf{(C3)},  that $[\beta,\alpha]$ is an invariant interval, and \textbf{(C1)}, that the map has a unique discontinuity point within this interval, regardless of whether the map is in the overlapping- or non-overlapping case or neither of these (e.g. when the jumps at $\beta+k(\alpha-\beta)$ are positive but there is an overlap in values of $\Phi\vert_{\lbrack\beta,w_1)}$ and $\Phi\vert_{(w_1,\alpha\rbrack}$).

 \begin{theorem}\label{unique_fixed_point}  Assume that conditions \textbf{(C1)} and \textbf{(C3)} hold and that $\Phi$ has a fixed point in $\lbrack\beta,w_1)$. By $w_f$ denote the largest fixed point in $\lbrack\beta,w_1)$. Then
 \begin{enumerate}
   \item if $\max\{\Phi(w): \ w\in (w_1,\alpha\rbrack\} < w_f$, then the rotation number $\varrho(\Psi,w)=0$ is unique and the system displays no MMOs;
   \item if $\max\{\Phi(w): \ w\in (w_1,\alpha\rbrack\} \geq w_f > \beta$, then there are subintervals of $[\beta,\alpha]$ of points with rotation number $0$, corresponding to orbits with no MMOs. However, if simultaneously $\Phi(\alpha)\geq \Phi(\beta)$, then there exists $\tilde{q}\in \mathbb{N}^*$ such that for every $q\geq \tilde{q}$, $\Phi$ admits also a periodic point $\hat{w}\in (w_f,\alpha)$ of period $q$, displaying MMOs.
 \end{enumerate}
 \end{theorem}

\begin{proof} The first result follows from the fact that every point $w\in[\beta,\alpha]\setminus\{w_1\}$ is mapped into $[\beta,w_f]$ after at most a few iterates, and, since $\Phi([\beta,w_f])\subset [\beta,w_f]$ and $\Phi$ is increasing therein, it is eventually attracted to one of the fixed points located in $[\beta,w_f]$. 

For the second result, the same argument applies to show the existence of subintervals of $[\beta,\alpha]$ with rotation number $0$ under the assumptions made. To establish that there is $\tilde{q}\in \mathbb{N}^*$ such that periodic points of every period greater than $\tilde{q}$ exist under the additional assumption that $\Phi(\alpha)\geq \Phi(\beta)$, it suffices to show that the rotation interval is of the form $[0,\delta]$ for some $\delta>0$. Since $w_f$ is a fixed point for the lower enveloping map $\Psi_l$, we clearly have $a(\Psi)=\varrho(\Psi_l)=0$. In contrast, the upper enveloppe $\Psi_r$ has no fixed points, and using a similar argument as in the proof of Proposition \ref{male2}, we show that $b(\Psi)=\varrho(\Psi_r)\geq \frac{\delta'}{\alpha-\beta}>0$ for some $\delta'>0$, which completes the proof.  
\end{proof}

\subsection{Evolution of the rotation number along a segment of $(d,\gamma)$ values}

In the previous subsections, we have investigated the rotation number or the rotation interval in various subcases existing under general assumption \textbf{(C1)}, i.e. the adaptation map features a unique discontinuity point in the interval $[\beta,\alpha]$. We illustrate numerically the dependence of the rotation number (thus also the MMO pattern fired) and its possible uniqueness on the values of parameters $d$ and $\gamma$.

 \begin{figure}[htbp]
 \centering
\includegraphics[width=\textwidth]{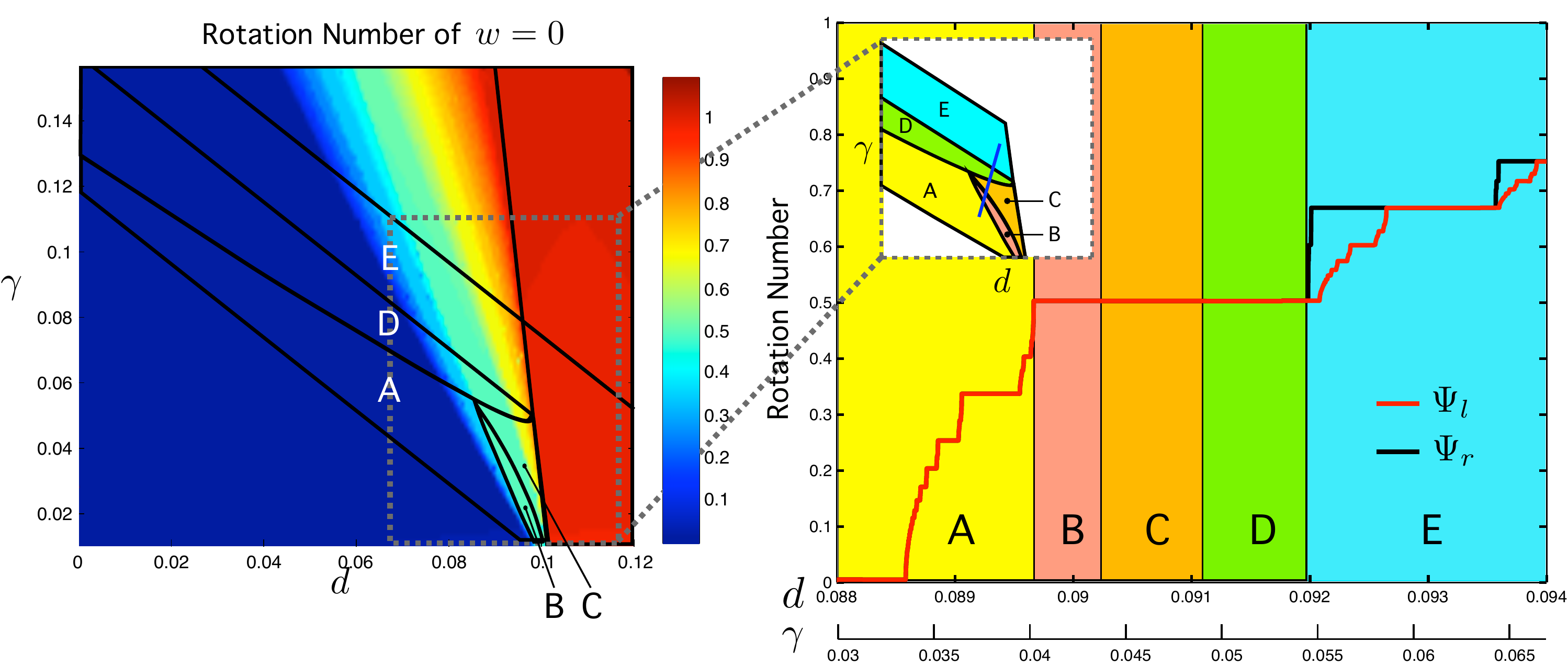}
\caption{Rotation numbers according to $(d,\gamma)$. Left panel : rotation number of the point $w=0$ together with the boundaries of the regions A to E corresponding to the different subcases when $w_1$ is the unique discontinuity of the adaptation map lying in the interval $[\beta,\alpha]$ (see text for more details). Right panel:  rotation numbers of the left and right lifts $\Psi_l$ and $\Psi_r$ associated with $\Phi$ for $(d,\gamma)$ varying along the  blue segment drawn in the inset. }
\label{RotNum_Path-d-Gamma}
\end{figure}

The left panel of Fig. \ref{RotNum_Path-d-Gamma} shows the rotation number of the adaptation map for a fixed initial condition and for $(d,\gamma)$ in $[0,0.12] \times [0.01,0.15]$. The various regions in the $(d,\gamma)$-plane corresponding to the different subcases studied above and already shown in Fig. \ref{Partition_d_gamma} are superimposed on the colormap. Regions A, B and C comprise the non-overlapping case, i.e. assumption \textbf{(C4)} is fulfilled, and general Theorem \ref{dynamics2} applies for $(d,\gamma)$ values in these regions. In particular, the rotation number of $\Phi$ is unique, i.e. does not depend on the initial condition.
\begin{itemize}
\item In region A, $\Phi(\alpha)<\Phi(\beta)<w_1$. Along certain paths in this region, Theorem \ref{new_devil} applies and the rotation number varies as illustrated in Fig. \ref{fig_Devil_Staircase}.  
\item In region B, $\Phi(\alpha)<w_1<\Phi(\beta)$, hence Proposition \ref{for2per} applies and ensures the existence of a period-2 orbit of $\Phi$, with rotation number equal to 1/2.
\item In region C, $w_1<\Phi(\alpha)<\Phi(\beta)$. This region may feature  a variety of different dynamics including all types of behavior arising in the other regions.  In the example in the right panel of Fig. \ref{fig_Devil_Staircase}, the unique rotation number is 1/2, but this value depends on the choice of $(d,\gamma)$, as can be seen in the left panel.
\end{itemize}
Regions D and E comprise the overlapping case and $\Phi$ may admit different rotation numbers depending on the initial condition. For $(d,\gamma)$ in these regions, the lift $\Psi$ associated with the adaptation map exhibits only negative jumps. The general Proposition \ref{Effective} applies, which ensures the existence of a rotation interval. Using the left and right lifts $\Psi_l$ and $\Psi_r$ associated with $\Phi$, one computes the endpoints  of the rotation interval and their evolution according to parameter $d$ (Proposition \ref{prop:overlap_devil} and Fig. \ref{Dev_St_Psir_Psil_Phi}). 
\begin{itemize}
\item In region D, $\alpha<w^*$ and $\Phi$ is piecewise increasing.  The rotation number is not uniquely defined in the general case. Nevertheless, along the particular chosen path in the parameter space $(d,\gamma)$ shown in the right panel of Fig. \ref{fig_Devil_Staircase}, $\Psi_l$ and $\Psi_r$ present the same rotation number $1/2$ and the rotation number of $\Phi$ does not depend on the initial condition. This particular simulation illustrates a way to demonstrate that the rotation number of the adaptation map is unique by showing that the rotation interval is reduced to a singleton.
\item In region E, $w^*<\alpha$.  The rotation numbers of $\Psi_l$ and $\Psi_r$ differ and the rotation interval of the adaptation map  varies with changes in  $(d,\gamma)$ within the region bounded by the black and red lines in the right panel of Fig. \ref{fig_Devil_Staircase}.
\end{itemize}

Note that the global Theorem \ref{unique_fixed_point} applies in all regions A to E. One may track the appearance and disappearance of the fixed points according to the values of $d$ and $\gamma$ together with the evolution of the rotation number or rotation interval. Outside of regions A to E, the structure of the lift is more complex due to the presence of additional discontinuity points. Yet, the numerical calculation of the rotation number can be performed for a given initial condition.

\section{A note on the case of two or more discontinuities} \label{sec:MoreDisc}
  One challenge in this study is related to the fact that the map under scrutiny, the adaptation map, is not known analy\-ti\-cal\-ly. Our mathematical analysis has covered in detail the cases of overlapping and non-overlapping maps with one discontinuity in the invariant interval. These situations do not cover all possible shapes of adaptation maps that can induce lifts with more discontinuity points; indeed, multiple discontinuities can yield a combinatorial explosion of cases with different combinations of possible jumps as well as maps that are non-monotone but with only positive jumps. While in these cases it is still possible to obtain  upper and lower bounds for the rotation set by computing the rotation numbers of the non-decreasing maps $\Psi_l$ and $\Psi_r$, defined in the same way  as in the overlapping case \textbf{(C4')}, it remains an open question to determine when every value within this interval corresponds to the rotation number of a given orbit, and it is not hard to find elementary examples for which this is false\footnote{We thank Micha\l{} Misiurewicz for interesting discussion on this topic.}. Thus the general, complete and precise characterization of the dynamics of the system is a complex and rich mathematical problem that raises several deep questions of iterates of interval maps with discontinuities.  In particular, we have seen that in the non-overlapping case \textbf{(C4)}, the rotation number allowed us to completely decode the MMO signature. A natural  extension of this work is thus to define for maps with more discontinuities a mathematical invariant (perhaps some vector of numbers) that would either provide the exact signature of each supported MMO pattern or allow calculation of how many points from a random orbit would be expected to fall into each continuity interval and hence how frequently a given number of small oscillations occurs between two consecutive spikes.

Let us conclude with the following exemplary result, which allows for multiple and even infinitely many intersections $w_i$ of the reset line $\{v=v_R\}$ with $\mathcal{W}^s$, assuming that only finitely many of them lie in the interval $(\beta,\alpha)$:
\begin{theorem}\label{GeneralOldHeavy} Suppose that $\Phi(\beta)>\beta$ and that there are finitely many discontinuity points of the map $\Phi$ in $(\beta,\alpha)$, all located in $(\beta,w^*)$.  Then the adaptation map $\Phi$ induces the rotation interval with the same properties as in Corollary \ref{Effective}. In particular, to every rational rotation number in the interior of this interval, there corresponds a periodic orbit, displaying regular MMOs.
\end{theorem}

The above theorem is straightforward once it is noted that the suitably defined lift $\Psi$ for the adaption map under the given assumptions is an old heavy map, as the maps studied in \cite{misiu}. Therefore, in particular, one can also derive conditions e.g. for periodic orbits of all possible periods exhibiting MMOs (with richer structure than what we considered earlier due to the additional discontinuities), and the corresponding regions in the space of reset parameters for specific models can be computed numerically, in the same way as in the previous subsection. We emphasize that due to the properties of the adaptation map (in particular, the fact that $\Phi(w_i,w_{i+1})=(\beta,\alpha)$ for the consecutive discontinuity points $w_i, w_{i+1}<w^*$), the lift $\Psi$ of $\Phi$ is very likely to be an old heavy map. Typically, in case of multiple discontinuities one can expect the rotation interval to cover the whole interval $[0,1]$ and the occurrence of periodic orbits of all periods and rich MMO structure.

\section{Discussion} \label{sec:Discussion}
Nonlinear bidimensional hybrid neuron models, which combine continuous subthreshold dynamics with a spike-related jump or reset condition, are easily defined and show an astonishingly rich mathematical phenomenology. A number of studies have already revealed their subthreshold dynamical properties~\cite{touboul:08}, investigated their spike patterns in the absence of any equilibrium state of the subthreshold dynamics~\cite{touboul-brette:09}, and highlighted their versatility~\cite{brette-gerstner:05,izhikevich:04,shlizerman2012neural} and capacity to reproduce neuronal dynamics~\cite{izhikevich:07,izhikevich-edelman:08,naud-macille-etal:08,touboul-brette:08}. The present paper and its companion~\cite{paper1} add to this body of works by studying (i) chaotic dynamics and period-incrementing structures, and (ii) oscillating solutions associated with multiple unstable equilibria. The latter led us to investigate the dynamics of a particular class of interval maps that feature both discontinuities and divergence of the derivative. Interestingly, in the presence of an unstable focus of the subthreshold dynamics, we have shown that the spike patterns fired may correspond to complex oscillations that combine action potentials (or bursts of action potentials) and subthreshold oscillations, trajectories known as MMOs or MMBOs in continuous dynamical system. 

In contrast to continuous dynamical systems, these forms of  complex oscillations can occur in hybrid systems with only two variables. Moreover, the mechanism of generation of these trajectories differs between these two models; in the hybrid case, MMOs result entirely  from  the topology of the invariant manifolds of the continuous-time dynamics. As such, these trajectories can occur in systems that lack timescale separation and based on a mapping approach, discrete  dynamical systems methods can be used to rigorously establish their existence and properties. One may however wonder if there exists a relationship between the two systems, and particularly it is tempting to interpret the hybrid system as the reduction of a differentiable multiple timescale system in a certain singular limit. The wide variety of MMOs (in particular the wild signatures encountered) produced with the reset mechanism indicates that such a differentiable system should be at least four-dimensional and the vector field should induce a highly complex return mechanism within the region of the phase space where small oscillations are generated (funnel). The construction of such a return mechanism for reproducing the same versatility in the MMOs signature in the differentiable case remains a challenging problem from the dynamical viewpoint, involving complex interactions between the different timescales.

To tune the model parameters to attain the regime studied in this work, we introduced a parameter $\gamma$, which  yields an attenuation of the adaptation variable during the reset.  This adjustment to the reset mechanism  
accounts for the  durations of spikes fired  (see \cite{paper1}). With this new parameter, the quartic model (and, we expect, all other models of the class, including the Izhikevich model~\cite{izhikevich:04} and the adaptive exponential~\cite{brette-gerstner:05}) can be tuned to achieve any of the cases we have identified. Therefore, our analysis provides useful information for tuning model parameters to achieve outputs fulfilling a list of qualitative and quantitative specifications. In particular, the ability to reproduce fine trajectories of MMOs may be useful when modeling neurons in situations in which synchronization is essential. Indeed, in neuroscience, it has been shown that in the pre\-sen\-ce of noise, small subthreshold oscillations support  the generation of precise and robust rhythmic spike patterns, as recorded in specific rhythmic pattern generators such as the inferior olive nucleus \cite{bernardo-foster:86,llinas-yarom:81,llinas-yarom:86}, in the stellate cells of the entorhinal cortex \cite{alonso-klink:93,alonso-llinas:89,jones:94}, and in the dorsal root ganglia \cite{amir-michaelis-etal:99,liu-michaelis-etal:00,llinas:88}.
A possible direction for future work would be to  go deeper into the analysis of the shape of the adaptation map of the adaptive exponential integrate-and-fire system to relate the presence and possible signature of MMOs to variations in biophysical parameters, following e.g.~\cite{touboul-brette:08}.

Another important direction related to the roles of model parameters would be to characterize the structural stability of trajectories and their possible bifurcations.  First works in that direction have been developed in~\cite{foxall2012contraction}: taking into account the infinite contraction of the trajectories in the voltage variable associated with the reset, the authors proposed to compute expansion or contraction exponents along transverse directions, providing a notion of stability of hybrid orbits that is more explicit than criteria on the shape of the adaptation map. It would be interesting to develop these methods in the cases of non-monotonic spiraling trajectories associated with the presence of MMOs. Alternatively, using models with simpler subthreshold dynamics, for instance linear or piecewise linear~\cite{jimenez2013locally,rotstein2012canard}, may allow for a derivation of an explicit expression of the reset maps, thus for fine characterization of the stability of the orbits.

At the level of the adaptation map, a question that is open in the overlapping case is to characterize the stability of orbits when the system has multiple possible rotation numbers. Indeed, even if the rotation interval is not a singleton, one often observes in simulations that  only one rotation number is actually realized. There are two typical reasons why this could occur:  either there is an attracting periodic orbit that attracts most  initial conditions or the system has an invariant measure $\mu$, absolutely continuous with respect to the Lebesgue measure, in which case the observed rotation number is just the average displacement $\Psi(w)-w$ with respect to the measure $\mu$. Nonetheless, rigorously establishing the existence of such a measure is  a challenge in most of systems arising from applications. In particular, we cannot use e.g. the classical Lasota-Yorke theorem (\cite{LasotaYorke}), since the derivative $\Phi^{\prime}$ diverges at the discontinuity points. On the other hand, for investigating stability of orbits  a possible approach would be to use and develop symbolic dynamics and kneading theory for such discontinuous interval maps. However, we emphasize that
in our  characterization of the orbits and the patterns of complex oscillations fired, rotation theory turned out to be the most useful tool since we have a unequivocal,
bidirectional link between the rotation number and the signature of the MMO (Theorem \ref{dynamics2}), which allows us to characterize situations in which the neuron shows regular spiking, MMO, bursting, MMBO or chaotic behavior. 

In these studies, we have made a crucial use of the planar nature of the system. MMOs will of course exist in higher dimensional hybrid dynamical systems, and analysis would require fine characterization of the invariant manifolds. The extension of the theory to higher dimensional systems would be particularly interesting from the computational neuroscience viewpoint for understanding the behavior of neuron networks in which several neurons driven by such dynamics are coupled and communicate at the times of the spikes.

\vspace{0.05cm}

\noindent {\bf Acknowledgements:}
J. Rubin was partly supported by US National Science Foundation awards DMS 1312508 and 1612913.  J. Signerska-Rynkowska was partly supported by Polish National Science Centre grant 2014/15/B/ST1/01710. 

\bigskip

\appendix
\section{Proofs of Theorems \ref{new_devil} and \ref{prop:overlap_devil}}\label{sec:appendix}

\vspace{0.3cm}

\noindent{\it{Proof of Theorem \ref{new_devil}}} A general theorem for continuous orientation-preserving circle maps is shown in~\cite{katok}, and is extended to the case of non-continuous orientation-preserving maps in~\cite{brette} and in~\cite{frrhodes2}. This theory is valid under non-degeneracy conditions on the dependence of the maps on the parameters. In particular, a general result on the monotone family of increasing lifts $\Psi_s$ indexed by a parameter $s\in [\lambda_1,\lambda_2]$ (in our case, $s=d$ or $\gamma$) can be shown under the assumption that the map $s\mapsto \Psi_s$ is increasing and continuous with respect to the Hausdorff topology of $H$-convergence, which is equivalent to uniform convergence at the continuity points (see~\cite{frrhodes2}), i.e. under the condition
	\begin{equation}\label{eq:Hconv}
	\underset{\varepsilon>0}{\underset{\tilde{w}\neq \alpha_{s_0}+k(\alpha_{s_0}-\beta_{s_0})}{\underset{s_0\in \lbrack \lambda_1,\lambda_2\rbrack}{\forall}}} \quad
	\underset{\xi>0}{\underset{\delta>0}{\exists}}\quad  \underset{w\in\mathbb{R}}{\underset{s\in \lbrack \lambda_1,\lambda_2\rbrack}{\forall}} \quad  \vert s-s_0\vert<\xi \land \vert w-\tilde{w}\vert<\delta  \implies  \vert \Psi_s(w)-\Psi_{s_0}(\tilde{w})\vert <\varepsilon
	\end{equation}
	where $\alpha_{s_0}+k(\alpha_{s_0}-\beta_{s_0})$, $k\in\mathbb{Z}$, denotes the discontinuity point of the lift $\Psi_{s_0}$.
		
As the reset parameter $d$ is increased, the map $\Phi$ is rigidly increased by the same amount. This particularly simple dependence of the map on $d$ yields precise control of how the dynamical features of the map vary with $d$. In particular, we note that the boundaries of the invariant interval $\alpha_d$ and $\beta_d$ are also simply translated as $d$ varies, and in particular the length $\theta:=\alpha_d-\beta_d$ of the invariant interval is constant. Moreover, we also observe that for any $d\in [d_1,d_2]$, the maps $\Phi_d$ have the same discontinuity point $w_{1,d}$, and the lifts $\Psi_d$ are continuous at points $w_{1,d}+k(\alpha_d-\beta_d)$, have positive jumps at $\alpha_d+ k(\alpha_d-\beta_d)$ and satisfy
$\Psi_d(w+\theta)=\Psi_d(w)+\theta$. So in fact all these lifts $\Psi_d$ can be seen as lifts of non-continuous invertible circle maps under the same projection $\mathfrak{p}: t\mapsto \exp(\frac{2\pi\imath t}{\theta}).$

	However, even if the map $\Phi_d$ is increasing with $d$, this is not necessarily the case for $\Psi_d$, because of the simultaneous fluctuation of the invariant interval. Indeed, when each lift $\Psi_d$ is obtained from $\Phi_{d}\big\vert_{\lbrack\beta_d,\alpha_d\rbrack}$ the relation $\Psi_{d_1}(w)<\Psi_{d_2}(w)$ for $d_1<d_2$ might be violated in the intervals $\lbrack \beta_{d_1},\beta_{d_2}\rbrack$, as at the point $\alpha_{d_2}$ we glue the right part of the graph of $\Phi_{d_2}\vert_{[\beta_{d_2},\alpha_{d_2}]}$ to its left part (shifted up by $\theta$). But noticing that under the additional condition $\Phi_d(\alpha_{d_2})<\Phi_d(\beta_{d_1})$ for any $d\in[d_1,d_2]$,  the interval $[\beta_{d_1},\alpha_{d_2}]$ constitutes a particular invariant interval in which the adaptation map $\Phi_d$ is piecewise increasing and non-overlapping, we can build well-behaved lifts $\tilde{\Psi}_d:\mathbb{R}\to\mathbb{R}$ based on the shape of the map $\Phi_d$ on this bigger invariant interval $[\beta_{d_1},\alpha_{d_2}]$. In contrast to $\Psi_d$, these new lifts are discontinuous at the points $w_{1,d}+k(\alpha_{d_2}-\beta_{d_1})$ (where they have positive jumps of amplitude $d_2-d_1$), in addition to their discontinuity at $\alpha_{d_2}+k(\alpha_{d_2}-\beta_{d_1})$, $k\in \mathbb{Z}$. The latter jump also remains positive under our assumption that $\Phi_d(\beta_{d_1})$ is strictly greater than $\Phi_d(\alpha_{d_2})$. Constructing lifts $\Psi_d$ on an enlarged invariant interval $[\beta_{d_1},\alpha_{d_2}]$ instead of $[\beta_d,\alpha_d]$ has the advantage of ensuring that the mapping $(w,d)\mapsto \tilde{\Psi}_d(w)$ is increasing in both variables.  Moreover, it has no effect on the dynamics, since any orbit of $\Phi_d$ with an initial condition in $[\beta_{d_1},\alpha_{d_2}]$ enters after a few iterations into the interval $[\beta_d,\alpha_d]$.  Since the orbits $\{\tilde{\Psi}^n_d(w)\}$ project $\mod (\alpha_{d_2}-\beta_{d_1})$ to the orbits $\{\Phi^n_d(w)\}$, we therefore have $\varrho(\tilde{\Psi}_d)=\varrho(\Psi_d)$.
	
Concluding the proof therefore only amounts to showing that the map $d\mapsto \tilde{\Psi}_d$ is continuous in the Hausdorff topology, which is very simple once it is noted, as mentioned above, that this property is equivalent to the uniform convergence at all points in the interior of $[\beta_{d_1},\alpha_{d_2}]\setminus \{w_{1,d}\}$ and that $\tilde{\Psi}_d-\tilde{\Psi}_{d'}=d-d'$ on this interval. Thus the mapping $\tilde{\rho}: d\mapsto \varrho(\tilde{\Psi}_d)$ has the properties listed in the theorem (compare with Theorem 2 in \cite{brette}) and consequently, the same holds for $\rho: d\mapsto \varrho(\Psi_d)$.  \hbx

We have noticed that while continuity of the lifts under the Hausdorff topology was always satisfied in our case, an additional assumption is necessary to ensure that the mapping $(s,w)\mapsto \Psi_s(w)$ (where $s$ denotes a parameter, here $d$ or $\gamma$) is increasing in both variables, which otherwise is not always true. We emphasize that even in situations in which this mapping is not increasing in both variables, the rotation number remains continuous under the $H$-convergence provided that the limit function $\Psi_{s_0}$ is strictly increasing, see \cite[Proposition 5.7]{frrhodes2}.

The plateaus of rotation number observed in the devil's staircase situation are a general property of our system, called \emph{locking} (see~\cite{frrhodes2} for precise definition of locking). 

\begin{remark}
	We observe that no condition beyond monotonicity of the lifts in $w$ and $d$ is required to show locking of the rational rotation number in the strictly non-overlapping case (i.e. $\Phi(\alpha)<\Phi(\beta)$), unlike the case of continuous circle maps. When $\Phi(\alpha)=\Phi(\beta)$, the lift $\Psi$ would be in fact a lift of an orientation preserving circle homeomorphism and thus locking of the rotation number at rational values requires that there is no conjugacy with rational rotation for such a map (see e.g. Propositions 11.1.10 and 11.1.11 in \cite{katok}).
\end{remark}

\vspace{0.3cm}
\noindent{\it{Proof of Theorem \ref{prop:overlap_devil}}}
The first part of the proof amounts to showing that the upper and lower envelopes of $\Psi_{d}$, denoted $\Psi_{d,l}$ and $\Psi_{d,r}$, are uniformly continuous in $d$ for $d\in \lbrack \lambda_1,\lambda_2\rbrack$. 

This regularity readily stems from the fact that $\Phi_d$ and $\Phi_{d_0}$ are simply shifted by the amount $d-d_0$. But as in the proof of Theorem~\ref{new_devil}, one needs to be careful about the variation of the invariant intervals $[\beta_d,\alpha_d]$ since these also have an additive relationship in $d$ (i.e. $\beta_d-\beta_{d_0}=d-d_0$ and similarly for $\alpha_d$). Thus close to the discontinuity, we do not have an additive relationship in $\Psi_d$ in general, but for the maps $\Psi_{d,l}$ and $\Psi_{d,r}$, we can prove even uniform continuity in $d\in \lbrack \lambda_1,\lambda_2\rbrack$:

\begin{equation}\label{cont_lower}
\forall \varepsilon>0, \exists \xi>0, \forall (d_1,d_2)\in \lbrack \lambda_1,\lambda_2\rbrack ^2,
 \vert d_1 - d_2\vert <\xi \ \implies\  \Vert \Psi_{d_1,l} -  \Psi_{d_2,l} \Vert_{\infty}<\varepsilon.
\end{equation}

We now fix $\varepsilon,\xi>0$ and $(d_1,d_2) \in [\lambda_1,\lambda_2]^2$ with $d_1-d_2<\xi$, and analyze the maps $\Psi_{d_1,l}$ and $\Psi_{d_2,l}$ in the interval $\lbrack \beta_{d_2},\alpha_{d_2}\rbrack$ without loss of generality, since the fact that  $\Psi_d(w+\theta)=\Psi_d(w)+\theta$ allows restricting the analysis to an arbitrary interval of length $\theta:=\alpha_d-\beta_d$.

We clearly have, for any $w\in \lbrack \beta_{d_1},\alpha_{d_2}\rbrack$: 
\[
\Psi^{d_1}_l(w)-\Psi^{d_2}_l(w)=d_1-d_2<\xi.
\]
For $w\in\lbrack \beta_{d_2},\beta_{d_1}\rbrack$, we find
\begin{eqnarray*}
	\Psi_{d_1,l}(w)&=&\min\{\Phi_{d_1}(w+\theta), \Phi_{d_1}(\beta_{d_1})\}, \\
	\Psi_{d_2,l}(w)&=&\Phi_{d_2}(w)\leq \Phi_{d_2}(\beta_{d_1})=\Phi_{d_1}(\beta_{d_1})-(d_1-d_2)<\Phi_{d_1}(\beta_{d_1}).
\end{eqnarray*}
We now distinguish between two cases depending on whether $\Psi_{d_1,l}(w)\geq \Psi_{d_2,l}(w)$ or not. When this inequality is true, we find
\begin{eqnarray*}
\vert \Psi_{d_1,l}(w)- \Psi_{d_2,l}(w)\vert &=& \Psi_{d_1,l}(w)- \Psi_{d_2,l}(w) \\
	&\leq& \Psi_{d_1}(\beta_{d_1})-\Psi_{d_2}(\beta_{d_2}) \\
	&\leq& \Psi_{d_2}(\beta_{d_1})-\Psi_{d_2}(\beta_{d_2}) + d_1-d_2\leq (1+\mathcal{C})\xi,
\end{eqnarray*}
where $\mathcal{C}:=\max\{(\Phi_{d})^{\prime}(w): \ w\in \lbrack \beta_{\lambda_1},\beta_{\lambda_2} \rbrack\}$ is actually a constant independent of $d$.
If, on the contrary, $\Psi_{d_1,l}(w)< \Psi_{d_2,l}(w)$, then we have
\[
\Psi_{d_2,l}(w)\leq \Phi_{d_2}(\beta_{d_1})=\Phi_{d_1}(\beta_{d_1})-(d_1-d_2)< \Phi_{d_1}(\alpha_{d_1})-(d_1-d_2)
\]
using the overlapping condition. Similarly, $\Psi_{d_1,l}(w)\geq \Psi_{d_1,l}(\beta_{d_2})=\Phi_{d_1}(\alpha_{d_2})$. Equipped with these estimates, we can compute  that $\vert \Psi_{d_1,l}(w)- \Psi_{d_2,l}(w)\vert  \leq (1+\tilde{\mathcal{C}})\xi$, where $\tilde{\mathcal{C}}:=\max\{(\Phi_d)^{\prime}(w): \ w\in \lbrack \alpha_{\lambda_1}, \alpha_{\lambda_2}\rbrack\}$ is independent of $d$, which proves  \eqref{cont_lower} for $\Psi_{d,l}$. Similar methods  will work for proving the property for upper-enveloping maps $\Psi_{d,r}$ concluding the proof of continuity of the mappings $d\mapsto a(\Psi_d)$ and $d\mapsto b(\Psi_d)$. 

Note that, in contrast to the proof of Theorem~\ref{new_devil}, we did not consider here the maps $\Phi_{d_1}$ and $\Phi_{d_2}$ on a common bigger invariant interval, e.g. $\lbrack \beta_{d_2},\alpha_{d_1}\rbrack$ for $d_1>d_2$, because such lifts would have positive jumps at $w_1$ and, consequently, would no longer correspond to heavy maps.

To prove the second statement, we consider again $(d_1,d_2) \in [\lambda_1,\lambda_2]^2$ such that $d_1>d_2$. For $d \in [d_2,d_1]$, we build the maps $\Psi_d$, $\Psi_{d,r}$ and $\Psi_{d,l}$ on the interval $[\beta_{d_2},\alpha_{d_2}]$. Note that $\Psi_{d_1}(w)-\Psi_{d_2}(w)=d_1-d_2>0$ for $w\in [\beta_{d_1},\alpha_{d_2}]\subset [\beta_{d_2},\alpha_{d_2}]$. The relation $\Psi_{d_1}(w)-\Psi_{d_2}(w)>0$  can only be violated in $[\beta_{d_2},\beta_{d_1}]$. However, $\Psi_{d_2}(w)\leq \Psi_{d_2}(\beta_{d_1})$ for $w\in [\beta_{d_2},\beta_{d_1}]$ since $\Psi_{d_2}$ is monotone increasing on this interval. On the other hand, depending on whether $w^*(d_1)\in [\beta_{d_2}+\theta,\alpha_{d_1}]$ or not, $\Psi_{d_1}$ in $[\beta_{d_2},\beta_{d_1}]$ is either monotone (non-decreasing or non-increasing) or has exactly one local extremum, namely $w^*(d_1)$. This yields
\[
\Psi_{d_1}(w)\geq \min\{\Psi_{d_1}(\beta_{d_2}),\Psi_{d_1}(\beta_{d_1}^-)\}
\]
for every $w\in [\beta_{d_2},\beta_{d_1}]$. Additionally, since $\Psi_{d_1}$ fulfills the overlapping condition,
\[
\Psi_{d_1}(\beta_{d_1}^-)> \Psi_{d_1}(\beta_{d_1}^+)=\Psi_{d_2}(\beta_{d_1})+d_1-d_2> \Psi_{d_2}(\beta_{d_1})
\]
and $\Psi_{d_1}(\beta_{d_1}^-)> \Psi_{d_2}(w)$ for every $w\in [\beta_{d_2},\beta_{d_1}]$.  Using an analogous argument for  $\Psi_{d_2}$, we obtain
\[
\Psi_{d_1}(\beta_{d_2})=\Psi_{d_2}(\beta_{d_2}^-)+d_1-d_2> \Psi_{d_2}(\beta_{d_2}^+)+d_1-d_2\geq \Psi_{d_2}(\beta_{d_1})
\]
due to (\ref{assump_overlap}). Thus $\Psi_{d_1}(\beta_{d_2})>\Psi_{d_2}(w)$ for every $w\in [\beta_{d_2},\beta_{d_1}]$. It follows that $\Psi_{d_1}(w)>\Psi_{d_2}(w)$ also in $[\beta_{d_2},\beta_{d_1}]$ and the mapping $d\mapsto \Psi_d$ is increasing. Now, by the definition of the enveloping maps $\Psi_{d,l}$ and $\Psi_{d,r}$, the fact that $\Psi_{d_2}<\Psi_{d_1}$ on $\mathbb{R}$ for $d_2<d_1$ implies that $\Psi_{d_2,r}<\Psi_{d_1,r}$ and $\Psi_{d_2,l}<\Psi_{d_1,l}$ on $\mathbb{R}$. Thus the maps $d\mapsto \Psi_{d,r}$ and $d\mapsto \Psi_{d,l}$ are increasing and  the statement about the devil's staircase follows. \hbx

\begin{remark}
To ensure that the mapping $t\mapsto \varrho(F_t)$ behaves as a devil's staircase for a continuous increasing family $\{F_t\}_{t\in [T_1,T_2]}$ of continuous  non-decreasing degree-one maps $F_t$, we also need to make sure that there exists a dense set $S\subset \mathbb{Q}$ such that, for $s\in S$, no map $F_{t}$ is conjugated to the rotation $\mathcal{R}_s$ by $s$ and that the map $t\mapsto \varrho(F_t)$ is not constant (see Proposition 11.1.11 in \cite{katok}). However, in practice, these two specific cases do not occur for any of the envelopes $\Psi_l$ and $\Psi_r$ of the adaptation map. \end{remark}

\end{document}